\documentclass[12pt]{article}

\usepackage{amsmath}
\usepackage{amsfonts}
\usepackage{amsthm}
\usepackage{graphicx}
\usepackage{overpic}
\usepackage{a4wide}
\usepackage{amssymb} 
\usepackage{pictex}
\usepackage{rotating}  
\usepackage{color}
\usepackage{cite}

\usepackage{esint} %for p.v. integral

\usepackage{comment}
\numberwithin{equation}{section}

\newtheorem{theorem}{Theorem}[section]
\newtheorem{proposition}[theorem]{Proposition}
\newtheorem{lemma}[theorem]{Lemma}

\newtheorem{Definition}[theorem]{Definition}

\newtheorem{Remark}[theorem]{Remark}
\newenvironment{remark}{\begin{Remark}\rm}{\end{Remark}}
\newtheorem{RHproblem}[theorem]{RH problem}

\newcommand{\A}{\mathbb{A}}
\newcommand{\C}{\mathbb{C}}
\newcommand{\D}{\mathbb D}
\newcommand{\Z}{\mathbb{Z}}
\newcommand{\N}{\mathbb{N}}
\newcommand{\R}{\mathbb{R}}
\newcommand{\T}{\mathbb{T}}
\renewcommand{\H}{\mathbb{H}}

\newcommand{\CC}{\mathcal C}
\newcommand{\DD}{\mathcal D}

\newcommand{\Scal}{\mathcal{S}}
\newcommand{\OO}{\mathcal{O}}

\renewcommand{\Re}{{\rm Re} \,}
\renewcommand{\Im}{{\rm Im} \,}

\renewcommand{\bar}{\overline}
\renewcommand{\tilde}{\widetilde}

\newcommand{\Div}{{\rm Div} \,}

\begin{document}
\title{A boundary value problem for conjugate conductivity equations}
%the homogeneous Grad--Shafranov equation}
\author{S.\ Chaabi, S.\ Rigat and F.\ Wielonsky}

\maketitle 

\begin{abstract}
We give explicit integral formulas for the solutions of planar conjugate conductivity equations
in a circular domain of the right half-plane with conductivity $\sigma(x,y)=x^{p}$, $p\in\Z^{*}$.
The representations are obtained via a Riemann-Hilbert problem on the complex plane when $p$ is even and on a two-sheeted Riemann surface when $p$ is odd. They involve the Dirichlet and Neumann data on the boundary of the domain. We also show how to make the conversion from one type of conditions to the other by using the so-called global relation. The method used to derive our integral representations could be applied in any bounded simply-connected domain of the right half-plane with a smooth boundary.
\end{abstract}
\\[\baselineskip]
{\bf MCS:} 35J25, 35Q15, 35J15
\section{Introduction and main results}
We seek explicit integral expressions for real-valued solutions of the Neumann/Dirichlet problem for the following two conductivity equations, considered in a bounded simply connected domain $\Omega$ of the right half-plane $\H=\{(x,y)\in\R^{2},~x>0\}$,
\begin{equation}\label{cond-eq}
\Div(\sigma\nabla u)=0,\qquad
\Div\Big(\frac{1}{\sigma}\nabla v\Big)=0,
\end{equation}
%or equivalently
%\begin{equation}\label{Ernst}
%\Delta u-\frac1x\partial_{x} u=0,\qquad
%\Delta v+\frac1x\partial_{x} v=0,
%\end{equation}
for the particular case of conductivity $\sigma(x,y)=x^{p}$, $p\in\Z$. 
%where the conductivity $\sigma(x)=1/x$ has no singularity. 
For a general conductivity $\sigma(x,y)>0$, $x\in\H$, the two equations are conjugate in the following sense. If the function $u$ satisfies the first equation in (\ref{cond-eq}) then the differential form $-\sigma\partial_{y}udx+\sigma\partial_{x}udy$ is closed, so by applying Poincar\'e theorem in the simply connected $\Omega$, there exists a function $v$ such that
\begin{equation}\label{CR}
\partial_{x}v=-\sigma\partial_{y}u,\qquad\partial_{y}v=\sigma\partial_{x}u.
\end{equation}
Then, $v$ solves the second equation in (\ref{cond-eq}). Conversely, for such a $v$, the differential form $\sigma^{-1}\partial_{y}vdx-\sigma^{-1}\partial_{x}vdy$ is closed and there exists a function $u$ that satisfies (\ref{CR}) and the first equation in (\ref{cond-eq}). Moreover, letting
$$\nu=(1-\sigma)/(1+\sigma)\in\R,$$
we have that (\ref{CR}) is equivalent to the so-called conjugate Beltrami equation,
\begin{equation}\label{Belt-conj}
\bar\partial f=\nu\bar{\partial f},
\end{equation}
for the complex-valued function $f=u+iv$  in $\Omega$. This equation differs substantially from the classical and extensively studied Beltrami equation $\bar\partial f=\nu{\partial f}$. The link between the conductivity equations (\ref{cond-eq}) and (\ref{Belt-conj}) was used in \cite{AP} to investigate the Dirichlet to Neumann map. For a detailed study of the conjugate Beltrami equation (\ref{Belt-conj}) see e.g. \cite{BLRR}. 

When $\sigma=1$, the equations in (\ref{cond-eq}) reduce to the Laplace equation and (\ref{CR})
to the Cauchy-Riemann equations. When $\sigma$ differs from 1, we may speak, by analogy, of a $\sigma$-harmonic function $u$ and a $\nu$-holomorphic function $f$.
For the conductivity $\sigma(x,y)=x^{p}$, $p\in\N^{*}$, the equations in (\ref{cond-eq}) rewrite as
\begin{align}\label{Ernst}
\Delta u+\frac{p}{x}\partial_{x} u=0,\\[10pt]\label{Grad-S}
\Delta v-\frac{p}{x}\partial_{x} v=0.
\end{align}
These equations are particular case of the general elliptic equation
\begin{equation}\label{GASPT}
\Delta u+\frac{\alpha}{x}\partial_{x}u=0,\qquad \alpha\in\R,
\end{equation}
which is the topic of the so-called generalized axially symmetric potential theory (GASPT), see \cite{Wein2}. The name comes from the relation with the Laplacian in spaces of higher dimensions. For instance, let $(r,\varphi,z)$ denotes the usual cylindrical coordinates in $\R^{3}$, $\Omega$ be a domain in the $(x,z)$-plane, and $\Omega'\subset\R^{3}$ be the domain obtained by rotation of $\Omega$ about the $z$-axis. Then, 
$u(r,z)$ is a solution of $\Delta u(r,z)+r^{-1}\partial_{r}u(r,z)=0$ in $\Omega$ if and only if $u(r,z)$ is harmonic in $\Omega'$. Actually, a similar relation holds for the more general equation (\ref{Ernst}) when $p\in\N^{*}$. Indeed, if $r=(x_{1}^{2}+\cdots+x_{p+1}^{2})^{1/2}$ and $
U(x_{1},\ldots,x_{p+2})=u(r,x_{p+2})$ then one checks easily that
$$\Delta U(x_{1},\ldots,x_{p+2})=\Delta u(r,x_{p+2})+\frac{p}{r}\partial_{r}u.$$
When considering a domain $\Omega$ with some simple geometry, this link with the Laplacian allows one to get explicit bases of solutions via the method of separation of variables. For instance, toroidal harmonics (i.e. Legendre functions with half-integer degrees, see e.g. \cite{Leb}) give a complete set of solutions to (\ref{GASPT}) when $\alpha=\pm1$, cf. \cite{Mor, All, Van}. The fact that generalized Legendre functions can be used to solve (\ref{GASPT}) for complex values of $\alpha$ appears in \cite{Chab, Chab-Rig}. 

The GASPT theory has been investigated by many authors, among whom Weinstein \cite{Wein0,Wein2,Wein,Wein3}, Vekua \cite{Vek}, Gilbert \cite{Gil1,Gil2,Gil3,Gil4}, Henrici \cite{Hen1,Hen2,Hen3}, Mackie \cite{Mac}, Ranger \cite{Ran}.
Equation (\ref{GASPT}) is related to a variety of problems in physics. In particular, when $p=1$, equation (\ref{Ernst}) can be interpreted as the linearized version of the Ernst equation in the case of a static spacetime. In \cite{LF}, this equation was studied in the unbounded domain consisting of the exterior of the real segment $(0,\rho_{0})$ in the right half-plane $\H$ and was related to the relativistic gravitational field produced by a rotating disk of matter. Equation (\ref{Grad-S}) with $p=1$ is related to the behavior of plasma in a tokamak. The goal of this device, which has a toroidal geometry, is to control location of the plasma in its chamber by applying magnetic fields on the boundary. From an assumption of axial symmetry, the problem is reduced to a plane section. Then, it follows from the Grad-Shafranov equation, a second-order elliptic nonlinear partial differential equation, see \cite{SHA, BL}, that the magnetic flux in the vacuum between the plasma and the circular boundary of the chamber satisfies the homogeneous equation (\ref{Grad-S}). Note, that in this instance, the conductivity equation (\ref{Grad-S}) takes place in an annular domain, that is a doubly connected domain.

%An interesting property of equations (\ref{Ernst})-(\ref{Grad-S}) is their relation with the Laplacian in $\R^{3}$. Indeed, Also,
%$u(x,y)$ is a solution of (\ref{Ernst}) in $\Omega$ if and only if
%$r^{-1}u(r,z)\cos\varphi$ is harmonic in $\Omega'$ 

Another interesting feature of equation (\ref{GASPT}) is that it satisfies recurrence and symmetry relations with respect to the coefficient $\alpha$. Namely, denoting by ($E_{\alpha}$) the equation (\ref{GASPT}), the following holds:
\begin{equation}\label{recur-eq}
u(x,y) \text{ solves } (E_{\alpha}) \text{ if and only if }x^{-1}\partial_{x}u(x,y)\text{ solves }(E_{\alpha+2}),
\end{equation}
%the same equation with $\alpha$ replaced with $\alpha+2$. 
\begin{equation}\label{sym-eq}
u(x,y)\text{ solves }(E_{\alpha})\text{ if and only if } 
x^{\alpha-1}u(x,y)\text{ solves }(E_{2-\alpha}),
\end{equation}
%the same equation with $\alpha$ replaced with $2-\alpha$.
see e.g. \cite{Wein}. Finally, note that, since
equation (\ref{GASPT}) is elliptic, we know, by general results, see e.g. \cite{GT}, that the Dirichlet problem with datas on the boundary $\partial\Omega$ of the domain has a unique solution, which is $C^{\infty}$ in $\Omega$.

In the present paper, we will stick, for simplicity, to the case of a domain $\Omega$ being the open disk $\DD_{a}=D(a,1)$, lying in the right half-plane $\H$,
%=\{z=x+iy,~x>0\}$, 
with $a>1$ a point on the positive real axis. As briefly explained at the end of Section \ref{RH}, the approach used to derive our integral representations could be adapted to a general Jordan domain $\Omega$ in $\H$ bounded by a smooth curve.

%Instead of (\ref{Ernst}), we consider the conjugate equation for a real-valued function $v$,
%\begin{equation}\label{CGS}
%\Delta v+{1\over x}{\partial v\over\partial x}=0,
%\end{equation}
%with Neumann condition on the boundary of the disk. 
%Also, for $k$ on $\CC$, we denote by $\widetilde k$ the other point on $\CC$ with the same imaginary part.

The first results of our study are the following Theorems \ref{expl-even} and \ref{expl-odd}, which give explicit integral representations for the solutions of equation (\ref{GASPT}) when $\alpha$  is an even and odd negative integer respectively. The method of proof follows the general scheme of the so-called unified transform method, see \cite{F}. It uses a Lax pair, from which one may define a function $\phi(z,k)$ depending on a spectral parameter $k$, and which can be characterized as the solution of a specific singular Riemann-Hilbert problem in the $k$-plane.
Solving this Riemann-Hilbert problem leads to the seeked integral representations.

In the sequel, for $z$ inside the disk $\DD_{a}$, we denote by $z_{r}$ the point on the circle $\CC_{a}=C(a,1)$ which has the same imaginary part as $z$ and lies to its right. 
\begin{theorem}\label{expl-even}
Let $u$ be a solution of the equation 
$$\Delta u+\alpha x^{-1}\partial_{x}u=0,\quad \alpha=-2m,\quad m\in\N,$$
in the domain $\DD_{a}$ with smooth tangential and (outer) normal derivatives $u_{t}$ and $u_{n}$ on the boundary $\CC_{a}$. Then $u$ admits the integral representation
\begin{equation}\label{u-expl}
u(z)=-\frac{1}{\pi}\Im\int_{(z,z_{r})}\big((k-z)(k+\bar{z})\big)^{m}J(z,k)dk+
2\Re(a_{r})+u(z_{r}),\qquad z\in \DD_{a},
\end{equation}
where integration is on the segment $(z,z_{r})$, and the quantity $a_{r}$ can be explicitly computed in terms of the tangential derivatives along $\CC_{a}$ of $u_{t}$ and $u_{n}$, up to order $m-1$, at $z_{r}$. The function $J(z,k)$ is given by
$$J(z,k)=-\int_{\CC_{a}}W(z',k),$$
where $W(z,k)$ is the differential form
\begin{align}\label{W-z-zbar}
W(z,k) & =\big((k-z)(k+\bar z)\big)^{-m-1}\left((k+\bar z)u_{z}(z)dz+(k-z)u_{\bar z}(z)d\bar z\right)\\
& = \big((k-z)(k+\bar z)\big)^{-m-1}\left((k-iy)u_{t}(z)+ixu_{n}(z)\right)ds,
\end{align}
with $z=x+iy$ and $ds$ the length element on $\CC_{a}$.
\end{theorem}
\begin{remark}
In the particular case $\alpha=m=0$, equation (\ref{GASPT}) is Laplace equation and the solutions $u$ are simply the functions harmonic in $\DD_{a}$. In this case, the expression in the right-hand side of (\ref{u-expl}) simplifies to
\begin{align*}
& \frac{1}{\pi}\Im\int_{(z,z_{r})}\int_{\CC_{a}}W(z',k)dk+u(z_{r})
= \frac{1}{\pi}\Im\int_{(z,z_{r})}\int_{\CC_{a}}\left(\frac{u_{z}(z)}{k-z}dz+\frac{u_{\bar z}(z)}{k+\bar z}d\bar z\right)dk+u(z_{r})\\
& = \frac{1}{\pi}\Im\int_{(z,z_{r})}-2i\pi u_{z}(k)dk+u(z_{r})=
-\int_{(z,z_{r})}u_{x}(x)dx+u(z_{r})
\end{align*}
which is indeed $u(z)$. Note that in the second equality we have applied Cauchy formula to the analytic function $u_{z}(z)$.
\end{remark}
\begin{theorem}\label{expl-odd}
Let $u$ be a solution of the equation 
$$\Delta u+\alpha x^{-1}\partial_{x}u=0,\quad \alpha=-2m+1,\quad m\in\N,$$ 
in the domain $\DD_{a}$ with smooth tangential and (outer) normal derivatives $u_{t}$ and $u_{n}$ on the boundary $\CC_{a}$. Then $u$ admits the integral representation
\begin{equation}\label{v-expl}
u(z)=-\frac{1}{2\pi}\Im\int_{\CC_{a}}\frac{\big((k-z_{r})(k+\bar{z}_{r})\big)^{m}J(z,k)}{
\sqrt{(k-z)(k+\bar z)}}dk+u(z_{r}),\qquad z\in \DD_{a},
\end{equation}
where integration is on the circle $\CC_{a}$, oriented counter-clockwise. The square root in the denominator has a branch cut along the segment $(-\bar z,z)$. 
Integration starts at $z_{r}$ where the square root is taken to be positive, and its determination is chosen so that it remains continuous along the path of integration. The function $J(z,k)$ is explicitly given in terms of the tangential and normal derivatives $u_{t}$ and $u_{n}$ and their derivatives of order up to $m-1$ along the boundary $\CC_{a}$. It can be written as a sum,
\begin{equation}\label{expr-D}
J(z,k)=J^{0}(z_{r},k)+\int\tilde W(z',k),\qquad k\in\CC_{a},
%-2\int\frac{(k-iy')u_{t}+ix'u_{n}}{\big((k-z')(k+\bar{z'})\big)^{-m+1/2}}ds,
\end{equation}
see (\ref{def-J0})--(\ref{D23}) for a precise definition of $J(z,k)$.
%where $z'(s)=x'(s)+iy'(s)$ and t
The path of integration in the above integral is the subarc from $z_{r}$ to $k$ on $\CC_{a}$. It lies in $\{\Im z\geq\Im z_{r}\}$ when 
$\Im k\geq\Im z_{r}$ and in $\{\Im z\leq\Im z_{r}\}$ when $\Im k\leq\Im z_{r}$.
% when $k$ belongs to the right half of $\CC_{a}$, and the union of arcs $(z_{r},\widetilde k)\cup(k,\widetilde k)$ when $k$ belongs to the left half of $\CC_{a}$. 
The definitions of $J^{0}$ and the differential form $\tilde W$ involve the square root 
$$\lambda(z',k)=\sqrt{(k-z')(k+\bar z')},$$
with a branch cut along the segment $(-\bar z',z')$. We choose the determination of $\lambda(z_{r},k)$ that behaves like $k$ (resp. $-k$) at infinity when $\Im k\geq\Im z_{r}$ (resp. $\Im k\leq\Im z_{r}$) and then keep a continuous determination of $\lambda(z',k)$ when $z'$ moves along the path of integration. 
\end{theorem}
\begin{remark}
Making use of the symmetry principles (\ref{recur-eq}) and (\ref{sym-eq}), one deduces easily from Theorems \ref{expl-even} and \ref{expl-odd} similar integral representations for the solutions of (\ref{Ernst}), that is in the case of a positive integer coefficient $\alpha$.
\end{remark}
The second result concerns the correspondance between the Dirichlet and Neumann data $u_{t}$ and $u_{n}$. Having a Lax pair is equivalent to the existence of a closed differential form. By Poincar\'e lemma, this leads to the vanishing of an integral on a closed contour, the so-called global relation. In the present case, it will follow from results in Section \ref{Lax-pair} that
\begin{equation}\label{rel-glob1}
\int_{\CC_{a}}[(k-z)(k+\bar z)]^{\alpha/2-1}
((k-iy)u_{t}(z)+ixu_{n}(z))ds=0,\qquad k\in\C\setminus(\DD_{a}\cup\DD_{-a}).
\end{equation}
It is sometimes conjectured that there always exists at least one global relation (i.e. a Lax pair) that allows for the recovering of one type of boundary values from the other one. In the case of an even coefficient $\alpha$, we show that, indeed, this correspondance can be performed explicitly from such a relation.
\begin{theorem}\label{corresp}
Assume $\alpha=-2(m-1)$, $m\in\N^{*}$, and $u_{t}$ is a given function in $L^{2}(\CC_{a})$. 
%Then the normal derivative $u_{n}$ in $L^{2}(\CC_{a})$ of the solution $u$ of the Dirichlet problem
Let $u_{n}$ be a function in $L^{2}(\CC_{a})$ such that the global relation (\ref{rel-glob1}) holds true.
%$$\Delta u-2\frac{(m-1)}{x}\partial_{x}u=0,\qquad u=u_{t} \text{ on }\CC_{a},$$
%can be explicitly recovered from the global relation
%\begin{equation}\label{eq-inject}
%\int_{\CC_{a}}\frac{xu_{n}(z)}{(k-z)^{m}(k+\bar z)^{m}}ds=
%\int_{\CC_{a}}\frac{(y+ik)u_{t}(z)}{(k-z)^{m}(k+\bar z)^{m}}ds,\qquad k\in\C\setminus(\DD_{a}\cup\DD_{-a}).
%\end{equation}
Then, $u_{n}$ is unique and can be explicitly recovered from that relation. A similar statement holds when $\alpha=2(m+1)$, $m\in\N$. In that case, (\ref{rel-glob1}) does not allow the reconstruction of $u_{n}$. A relation that works is obtained by integrating another differential form, see (\ref{defW2}).
%from (\ref{eq-inject}). 
\end{theorem}
\begin{remark}
Making use of the equation conjugate to (\ref{GASPT}), that is changing $\alpha$ into $-\alpha$, we derive the converse reconstruction of the Dirichlet data $u_{t}$ from the Neumann data $u_{n}$, still when $\alpha\in 2\Z$ is an even coefficient. In Theorem \ref{corresp}, it is actually not necessary to assume smoothness of the functions $u_{t}$ and $u_{n}$ on the boundary circle $\CC_{a}$. Hence, we only assume these functions to be $L^{2}$ on $\CC_{a}$.
\end{remark}
Apparently, the reconstruction of $u_{n}$ from (\ref{rel-glob1}), or a similar relation, is not completely straightforward. It may be conjectured that the reconstruction of $u_{n}$ should be possible by an integral transform, like one of the Abel type, see e.g. \cite{TV}.
We were not able to find such an integral transform in the present case.
Our method uses rather the symmetry involved in the problem and the property of a particular linear differential equation stemming from the computation of a contour integral by Cauchy formula.
%solving a Dirichlet problem (for the usual Laplacian), and two linear systems of equations whose matrices are given by wronskians of orthogonal polynomials. 
%In particular, we were not able to find an integral transform that inverts the left-hand side of (\ref{eq-inject}).

In Section \ref{Lax-pair} we recall the notion of Lax pairs and compute such pairs for the equation (\ref{GASPT}). We also briefly discuss the use of a Lax pair, or the related global relation, in deriving an explicit correspondance between Dirichlet and Neumann data in the simple case of the Laplace equation. The study of two specific Riemann-Hilbert problems, leading to the proofs of Theorems \ref{expl-even} and \ref{expl-odd}, is performed in Sections \ref{RH}. The proof of Theorem \ref{corresp} is displayed in Section \ref{global}. 
%IN THE LAST SECTION
\section{Lax pairs and closed differential forms}\label{Lax-pair}
A Lax pair for a partial differential equation $P(u)=0$ is a pair of ordinary differential equations, for a function $\phi$ related to $u$, which are compatible precisely when $P(u)=0$. % if and only if $u$ satisfies (\ref{GASPT}).
The general equation (\ref{GASPT}) admits such a Lax pair.
%, that is a pair of differential equations, for a single function $\phi$, which are compatible if and only if $u$ satisfies (\ref{GASPT}). 
Indeed, writing (\ref{GASPT}) with respect to the complex variables $z$ and $\bar{z}$, we get
\begin{equation}\label{CCGS}
u_{z\bar{z}}+\frac{\alpha}{2(z+\bar{z})}(u_{z}+u_{\bar{z}})=0,
\end{equation}
and a possible way to find a Lax pair is by rewriting (\ref{CCGS}) in the form
\begin{equation}\label{LP1}
(f(z,\bar z)u_{\bar z})_{z}+(g(z,\bar z)u_{z})_{\bar z}=0,
\end{equation}
where the functions $f$ and $g$ have to be determined. Expanding (\ref{LP1}) and comparing with (\ref{CCGS}), we get
\begin{equation}\label{LP2}
f_{z}=\frac{\alpha}{2(z+\bar z)}(f+g),\qquad g_{\bar z}=\frac{\alpha}{2(z+\bar z)}(f+g).
\end{equation}
Differentiating the first equation with respect to $\bar z$, the second one with respect to $z$, and adding, we obtain
$$
(f+g)_{z\bar z}=\frac{\alpha}{2(z+\bar z)}((f+g)_{z}+(f+g)_{\bar z})
-\frac{\alpha}{(z+\bar z)^{2}}(f+g),
$$
so that $f+g$ satisfies the adjoint equation of (\ref{CCGS}) (see e.g. \cite[Chapter 7]{FOL} for the notion of the adjoint equation). We seek solutions in the form 
$$(f+g)(z,\bar z)=(z+\bar z)A(z)B(\bar z).$$
Plugging that in the previous equation leads to
$$(z+\bar z)A_{z}B_{\bar z}=(\alpha/2-1)(AB_{\bar z}+A_{z}B),$$
and thus
$$
-(\alpha/2-1)A/A_{z}+z=k=(\alpha/2-1)B/B_{\bar z}-\bar z,
$$
where $k$ may be any complex number. It is a new, additional parameter in the problem, the so-called spectral parameter.
Solving for the two equations, we get, as possible solutions,
$$A(z)=(k-z)^{\alpha/2-1},\qquad B(\bar z)=-(k+\bar z)^{\alpha/2-1}.$$
In view of (\ref{LP2}), we may thus choose for $f$ and $g$,
$$
f(z,\bar z)=(k-z)^{\alpha/2}(k+\bar z)^{\alpha/2-1},\qquad
g(z,\bar z)=-(k-z)^{\alpha/2-1}(k+\bar z)^{\alpha/2}.
$$
Hence, (\ref{CCGS}) is equivalent to
$$
\big((k+\bar{z})^{\alpha/2-1}(k-z)^{\alpha/2}u_{\bar{z}}\big)_{{z}}-
\big((k+\bar{z})^{\alpha/2}(k-z)^{\alpha/2-1}u_{z}\big)_{\bar{z}}
=0.
$$
%where $k\in\C$ is an additional parameter, the so-called spectral parameter. 
%One may then consider a potential $M(z,\bar{z},k)$ such that
%$$M_{z}=(\bar{z}+k)^{\alpha}(z-k)^{\alpha-1}u_{z},\qquad
%M_{\bar{z}}=-(\bar{z}+k)^{\alpha-1}(z-k)^{\alpha}u_{\bar{z}}.$$
%Writing $M$ in the form $M(z,\bar{z},k)=(\bar{z}+k)^{\alpha}(z-k)^{\alpha}\mu$, the previous equations become
%$$\alpha\mu+(z-k)\mu_{z}=u_{z},\qquad
%\alpha\mu+(\bar{z}+k)\mu_{z}=-u_{\bar{z}}.$$
This last equation is equivalent to the compatibility of the two ordinary differential equations
\begin{equation}\label{Lax}
\phi_{z}(z,k)=(k+\bar{z})^{\alpha/2}(k-z)^{\alpha/2-1}u_{z}(z),\qquad
\phi_{\bar z}(z,k)= (k+\bar{z})^{\alpha/2-1}(k-z)^{\alpha/2}u_{\bar{z}}(z),
\end{equation}
which thus gives a Lax pair for (\ref{CCGS}).
%with 
%$$\lambda(z,k)=\sqrt{\frac{k-z}{k+\bar{z}}}.$$
%\\[\baselineskip]
%CAS $\alpha=0$
%\\[\baselineskip]
Equivalently, we may express the property of the Lax pair as the closedness of a differential form.
\begin{proposition}\label{Prop-W}
The function $u$ satisfies (\ref{CCGS}) in a bounded, simply connected domain $\Omega$ of $\H$ if and only if
%, for any $k\in\C$ such that $k$ and $-\bar k$ do not belong to $\Omega$, 
the differential form
\begin{equation}\label{defW}
z\mapsto W(z,k)=\big[(k-z)(k+\bar{z})\big]^{\alpha/2-1}\big[(k+\bar{z})u_{z}(z)dz+(k-z)u_{\bar{z}}(z)d\bar{z}\big]
\end{equation}
is closed in $\Omega$. Note that, when $\alpha\in 2\N^{*}$, the differential form has no singularities in $\Omega$ and $k$ can be any complex number. Otherwise, for $\alpha\in\R\setminus 2\N^{*}$, $W(z,k)$ has either a pole or a branch point at $k$ and
 $-\bar k$, so that $k$ should lie outside $\Omega$ and $-\bar\Omega$.
 % if one of these points belongs to $\Omega$.
\end{proposition}
Making use of the symmetry relation (\ref{sym-eq}), we derive a second Lax pair, namely
\begin{align*}%\label{Lax2}
\phi_{z}(z,k) & =(k+\bar{z})^{1-\alpha/2}(k-z)^{-\alpha/2}((z+\bar z)^{\alpha-1}u)_{z}(z)\\
& = (k+\bar{z})^{1-\alpha/2}(k-z)^{-\alpha/2}(z+\bar z)^{\alpha-2}
[(z+\bar z)u_{z}(z)+(\alpha-1)u(z)],
\\
\phi_{\bar z}(z,k) & = (k+\bar{z})^{-\alpha/2}(k-z)^{1-\alpha/2}((z+\bar z)^{\alpha-1}u)_{\bar{z}}(z)
\\
& = (k+\bar{z})^{-\alpha/2}(k-z)^{1-\alpha/2}(z+\bar z)^{\alpha-2}
[(z+\bar z)u_{\bar z}(z)+(\alpha-1)u(z)],
\end{align*}
and the related closed form
\begin{multline}\label{defW2}
z\mapsto W(z,k) 
=\big[(k-z)(k+\bar{z})\big]^{-\alpha/2}x^{\alpha-2}\\ 
\cdot\big[(k+\bar{z})(2xu_{z}(z)+(\alpha-1)u(z))dz+(k-z)(2xu_{\bar{z}}(z)+(\alpha-1)u(z))d\bar{z}\big].
\end{multline}
This second differential form will be useful in Section \ref{global}.

We end this section with a brief discussion on the possible use of a Lax pair for the derivation of an explicit correspondance between different types of boundary data. It is easy to check that not all Lax pairs can achieve this goal. For instance, if we consider (\ref{Lax}) when $m=0$, we get
$$\phi_{z}(z,k)=u_{z}(z)/(k-z),\qquad
\phi_{\bar z}(z,k)= u_{\bar{z}}(z)/(k+\bar{z}),$$
which is a Lax pair for the Laplace equation $\Delta u=0$. If we choose as a domain $\Omega$ the unit disk $\D$, the differential form $W(z,k)$ from Proposition \ref{Prop-W} is closed in $\D$ and thus we get, in terms of the tangent and normal derivatives $u_{t}$ and $u_{n}$ on the unit circle,
\begin{equation}\label{eq-in-D}
\int_{\partial\D}\frac{xu_{n}(z)}{(k-z)(k+\bar z)}ds=
\int_{\partial\D}\frac{(y+ik)u_{t}(z)}{(k-z)(k+\bar z)}ds,\qquad k\in\C\setminus\D.
\end{equation}
Assume that the tangent derivative $u_{t}$ is known, hence also the integral in the right-hand side, equal to some function $\psi(k)$. Decomposing the real function $xu_{n}(z)=g(z)+\bar g(1/z)$, with $g(z)$ analytic in $\D$, $g(0)\in\R$, it is readily checked that (\ref{eq-in-D}) rewrites as 
$$\bar g(1/k)-g(-1/k)=k\psi(k),$$
which allows one to recover only the half of $g$, namely the imaginary parts of its even Taylor coefficients and the real parts of the odd ones.

On the contrary, when considering the domain $\DD_{a}=D(a,1)$ instead of $\D$, we will see in Section \ref{global} that the Lax pair in (\ref{Lax}) is sufficient to recover the Dirichlet to Neumann correspondance.
%%%%%%%%%%%%%%%%
\section{Characterization by a Riemann--Hilbert problem}\label{RH}
The goal of this section is to prove Theorems \ref{expl-even} and \ref{expl-odd}. \subsection{Case of a negative even integer coefficient $\alpha=-2m$, $m\in\N$}
\begin{proof}[Proof of Theorem \ref{expl-even}]
The differential form $W(z,k)$ rewrites as
\begin{equation}\label{defW-pair}
W(z,k)=\big((k-z)(k+\bar{z})\big)^{-m-1}\big((k+\bar z)u_{z}(z)dz+
(k-z)u_{\bar{z}}(z)d\bar{z}\big).
\end{equation}
Equations (\ref{Lax}) can be written as $d\phi=W$ and we thus construct a function $\phi$ of the form
\begin{equation}\label{def-phi-even}
\phi(z,k)=\int_{z_{r}}^{z} W(z',k),
\end{equation}
where the path of integration needs to be defined. For $k$ such that $\Im k\geq\Im z$, we integrate from $z_{r}$ to $z$ following the lower part $\CC_{low}$ of the circle $\CC_{a}$ and then the segment from $z_{l}$ to $z$. For $k$ such that $\Im k\leq\Im z$, we integrate from $z_{r}$ to $z$ following the upper part of the circle $\CC_{up}$ and then again the segment from $z_{l}$ to $z$, see Figure \ref{path-even}.
\begin{figure}[htb]
\centering
\def\svgwidth{14cm}
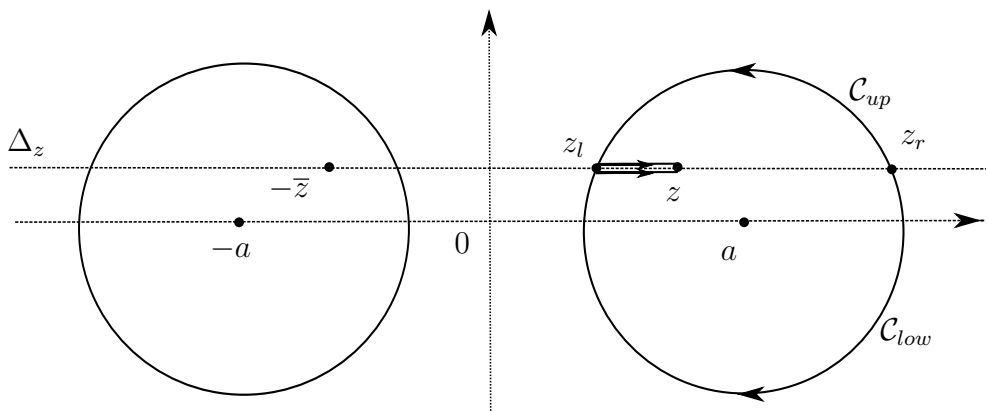
\caption{The paths of integration $\gamma_{1}$ and $\gamma_{2}$ from $z_{r}$ to $z$, respectively along $\CC_{up}$ and $\CC_{low}$ for the definition (\ref{def-phi-even}) of $\phi(z,k)$.}
\label{path-even}
\end{figure}
This defines $\phi(z,k)$ as an analytic function of $k$ outside of the line $\Delta_{z}$ through $z$ and $-\bar z$ where it may possibly have a jump. 
%Note that the bracketed terms give the polar parts of $\phi(z,k)$ at $z$ and $z_{r}$ while the last integral, as a function of $k$, has logarithmic singularities at these points.
We remark from (\ref{defW-pair}) and (\ref{def-phi-even}) that the function $k\mapsto\phi(z,k)$ has poles of order $m$ at $k\in\{-\bar z_{r},-\bar z,z,z_{r}\}$ if $m>0$ and logarithmic singularities if $m=0$. 
%Its polar parts at these points can be derived from (\ref{expans-phi}) and 
It also satisfies the symmetry relation
\begin{equation}\label{sym-phi-even}
\phi(z,-\bar k)=-\bar{\phi(z,k)}.
\end{equation}
In the sequel, we use the notations $\phi^{+}(z,k)$ and $\phi^{-}(z,k)$ for the limit values of $\phi(z,k)$ when $k$ tends to the left and right of a given arc, according to its orientation. The fact that the differential form $W(z',k)$ is closed in the disk $\DD_{a}$ when $k$ lies outside of that disk implies that there is no jump
$$J(z,k):=\phi^{+}(z,k)-\phi^{-}(z,k),$$
of $\phi(z,k)$ on the part of $\Delta_{z}$ outside of $\DD_{a}$ and $-\DD_{a}$. For $k$ inside the disks, we have
\begin{alignat*}{2}
J(z,k) & =-\int_{\CC_{a}}W(z',k),& \qquad k\in (z,z_{r})\cup (-\bar z_{r},-\bar z),\\
J(z,k) & =0, & \qquad k\in(z_{l},z)\cup (-\bar z,-\bar z_{l}).
\end{alignat*}
For computing the jump on $(-\bar z_{r},-\bar z)$, we have used the symmetry (\ref{sym-phi-even}). For the jump on $(z_{l},z)$, we have deformed the two paths of integration into the segment $(z_{r},z)$ which is possible by closedness of the differential form $W(z',k)$.
%On the other hand, $\phi(z,k)$, as a function of $k$, has at $z$ and $-\bar z$ logarithmic singularities if $m=0$ and polar singularities of order $m$ if $m>0$. 
Note that $k\mapsto J(z,k)$ has no singularity in $z$ and $z_{r}$. In $z$ it is clear, while in $z_{r}$, it can be seen by performing $m$ integrations by parts of $W(z',k)$ on the closed contour $\CC_{a}$, see also \cite[Section 4.4]{Gak} for a similar computation. Concerning the function $k\mapsto\phi(z,k)$, as said before, it has 
poles of order $m$ at $\{z, z_{r}, -\bar z, -\bar z_{r}\}$ if $m>0$ and logarithmic singularities if $m=0$. 
Moreover, at infinity, it behaves like
\begin{align*}
\phi(z,k)\to k^{-2m-1}(u(z)-u(z_{r})),\quad\text{as }k\to\infty.
%\phi(z,k)\to k^{-2m-1}(u(z)-u(z_{1})),\quad\text{as }k\to\infty\text{ with }\Im k<\Im z.
\end{align*}
If $m>0$, we renormalize the problem at infinity by defining
\begin{align}\label{norm-phi}
\tilde\phi(z,k)=((k-z)(k+\bar z))^{m}\phi(z,k),\\\label{norm-J}
\tilde J(z,k)=((k-z)(k+\bar z))^{m}J(z,k).
\end{align}
Then, the  function $k\mapsto\tilde\phi(z,k)$ behaves like $k^{-1}(u(z)-u(z_{r}))$ at infinity and is regular at $z$ and $-\bar z$, except for logarithmic singularities when $m=0$. 
It still has polar singularities at $z_{r}$ and $-\bar z_{r}$. Let us denote by $\tilde\phi_{z_{r},-\bar z_{r}}(z,k)$ its polar part at these points, that is the sum of the terms of negative degrees in the Laurent expansions of $\tilde\phi(z,k)$ at $z_{r}$ and $-\bar z_{r}$. The function $\tilde\phi-\tilde\phi_{z_{r},-\bar z_{r}}$ is analytic outside of the segments $(z,z_{r})$ and $(-\bar z_{r},-\bar z)$ where it has the jump $\tilde J(z,k)$. It has at most logarithmic singularities at $\{z,z_{r},-\bar z,-\bar z_{r}\}$ and vanishes at infinity.

These properties completely determine the function $\tilde\phi-\tilde\phi_{z_{r},-\bar z_{r}}$. Indeed, if there were another function with these properties, their difference would be entire and vanishing at infinity, hence the zero function. Thanks to the Plemelj formula, we have an integral expression for $\tilde\phi(z,k)-\tilde\phi_{z_{r},-\bar z_{r}}(z,k)$, namely,
\begin{equation}\label{explicit-phi}
\tilde\phi(z,k)-\tilde\phi_{z_{r},-\bar z_{r}}(z,k)=
\frac{1}{2i\pi}\int_{(-\bar z_{r},-\bar z)\cup(z,z_{r})}\frac{\tilde J(z,k')}{k'-k}dk'.
\end{equation}
Denoting by $a_{r}$ and $a_{-r}$ the residues of $\tilde\phi(z,k)$ at $k=z_{r}$ and $k=-\bar z_{r}$ respectively, and equating the coefficients of $k^{-1}$ in the expansion of (\ref{explicit-phi}) at infinity, we get
\begin{align*}
u(z)-u(z_{r}) & = a_{r}+a_{-r}-\frac{1}{2i\pi}\int_{(-\bar z_{r},-\bar z)\cup(z,z_{r})}\tilde J(z,k')dk'\\
& =2\Re(a_{r})-\frac{1}{\pi}\Im\int_{(z,z_{r})}\tilde J(z,k')dk',
\end{align*}
where we have used the symmetry relation $\tilde\phi(z,-\bar k)=-\bar{\tilde\phi(z,k)}$.
It remains to show that the residue $a_{r}$ can be explicitly computed from the knowledge of the derivatives $u_{t}$ and $u_{n}$ on $\CC_{a}$. From (\ref{norm-phi}), it is sufficient to know the polar part of $\phi(z,k)$ at $z_{r}$. 
To compute this polar part, we first rewrites $W(z,k)$, defined in (\ref{defW-pair}), by expressing the complex derivatives in terms of the tangential and (outer) normal derivatives on the circle $\CC_{a}$,
\begin{equation}\label{polar}
u_{z}dz =\frac{1}{2}(u_{t}+iu_{n})ds,\qquad
u_{\bar z}d\bar z =\frac{1}{2}(u_{t}-iu_{n})ds,
\end{equation}
with $ds$ the length element on $\CC_{a}$. We get, with $z=x+iy\in\CC_{a}$,
%the differential form $W(z,k)$, defined in (\ref{defW}), rewrites on $C$ as
\begin{align*}%\label{Wtn}
W(z,k) & =\big((k-z)(k+\bar z)\big)^{-m-1}\left((k-iy)u_{t}(z)+ixu_{n}(z)\right)ds,\\
& = (k-z)^{-m-1}w(z,k)dz
\end{align*}
where 
$$w(z,k):=(k+\bar z)^{-m-1}\left((k-iy)u_{t}(z)+ixu_{n}(z)\right)\tau^{-1}(z)$$ 
and $\tau(z)$ denotes the unit vector tangent to $\CC_{a}$ at the point $z$. Let us set
\begin{equation}\label{def-der}
\tilde\partial_{t}f:=\tau^{-1}(z)\partial_{t}f
\end{equation}
for a function $f$ on $\CC_{a}$. Because of analiticity, 
$$\partial_{z}(k-z)^{-j-1}dz=\partial_{t}(k-z)^{-j-1}ds=\tilde\partial_{t}(k-z)^{-j-1}dz,\quad j=m-1,\ldots,0.$$ 
Hence, performing $m$ integrations by parts on the integral in (\ref{def-phi-even}), we obtain
\begin{multline}
\label{expans-phi}
\phi(z,k)=c_{0}\left[(k-z')^{-m}w(z',k)\right]_{z_{r}}^{z}+\cdots \\
+
c_{m-1}\left[(k-z')^{-1}\tilde\partial_{t}^{(m-1)}w(z',k)\right]_{z_{r}}^{z}
-c_{m-1}\int_{z_{r}}^{z}
(k-z')^{-1}\tilde\partial_{t}^{(m)}w(z',k)dz',
\end{multline}
where
%$$c_{0}=(m-1/2)^{-1},\ldots,c_{m-1}=(-1)^{m-1}\Gamma(1/2)/\Gamma(m+1/2).$$
$$c_{j}=(-1)^{j}\Gamma(m-j)/\Gamma(m+1),\qquad j=0,\ldots,m-1.$$
The polar part of $\phi(z,k)$ at $z_{r}$ can be read in the bracketed terms in (\ref{expans-phi}) as
$$-c_{0}\sum_{j=0}^{m-1}\frac{\partial_{k}^{(j)}w(z_{r},z_{r})}{j!(k-z_{r})^{m-j}}-\cdots
-c_{m-1}\frac{\tilde\partial_{t}^{(m-1)}w(z_{r},z_{r})}{k-z_{r}},$$
where $\partial_{k}$ denotes the operator of differentiation with respect to the variable $k$.
This finishes the proof of Theorem \ref{expl-even}.
\end{proof}
%%%%%%%%%%%%%%%%%%%%%%%%%%%%%%%%%%%%%
\subsection{Case of a negative odd integer coefficient $\alpha=-2m+1$, $m\in\N$}
\begin{proof}[Proof of Theorem \ref{expl-odd}]
The differential form (\ref{defW}) now rewrites as
\begin{equation}\label{defW-imp}
W(z,k)=\big((k-z)(k+\bar{z})\big)^{-m-1/2}\big((k+\bar z)u_{z}(z)dz+
(k-z)u_{\bar{z}}(z)d\bar{z}\big).
\end{equation}
%with 
%$$\lambda(z,k)=\sqrt{\frac{k-z}{k+\bar{z}}}.$$
The novelty with respect to the even coefficient case is that $W(z,k)$ involves the square root 
$$\lambda(z,k)=\sqrt{(k-z)(k+\bar z)}$$
which, as a function of $k$, is defined on a Riemann surface $\Scal_{z}$ of genus 0, consisting of two copies of $\C$, denoted by $\Scal_{z,1}$ for the upper sheet, and by $\Scal_{z,2}$ for the lower sheet. The two sheets are glued together along a branch cut from $z$ to $-\bar z$, that we choose to be the horizontal segment $(-\bar{z},z)$. We denote by $\lambda_{1}$ and $\lambda_{2}$ the determination of the square root $\lambda$, as a function of $k$, on the upper and lower sheets of $\Scal_{z}$ where we assume that 
\begin{alignat}{2}
&\lambda_{1}(z,k)= k(1+\OO(1/k)),& \quad&\text{as $k\to\infty_{1}$ on the upper sheet $\Scal_{z,1}$},\label{inf1}\\
&\lambda_{2}(z,k)= -k(1+\OO(1/k)), & \quad&\text{as $k\to\infty_{2}$ on the lower sheet $\Scal_{z,2}$}.\label{inf2}
\end{alignat}
We also denote by $W_{1}$ and $W_{2}$ the values of the differential form $W$ corresponding to the determinations $\lambda_{1}$ and $\lambda_{2}$ of the square root. For future use, we remark that the function $\lambda(z,k)$ and the differential form $W(z,k)$ satisfy the following symmetry relations:
\begin{equation}\label{sym-W}
\lambda(z,-\overline k)=-\overline{\lambda(z,k)},\qquad W(z,-\overline k)=\overline{W(z,k)}.
\end{equation}

As in the previous section, we construct a function $\phi$ of the form
\begin{equation}\label{def-phi}
\phi(z,k)=\int_{z_{r}}^{z} W(z',k),
\end{equation}
where the path of integration needs to be defined. 
Similarly to the method applied in \cite{LF}, we define, for each $z$, the function $k\to\phi(z,k)$ as a map from the Riemann surface $\Scal_{z}$ to $\C$. The path of integration in (\ref{def-phi}) is chosen to be $\gamma_{1}$ when $k\in\Scal_{z,1}$ and $\gamma_{2}$ when $k\in\Scal_{z,2}$, see Figures \ref{regionI} and
\ref{regionII}. 
\begin{figure}[htb]
\centering
\def\svgwidth{14cm}
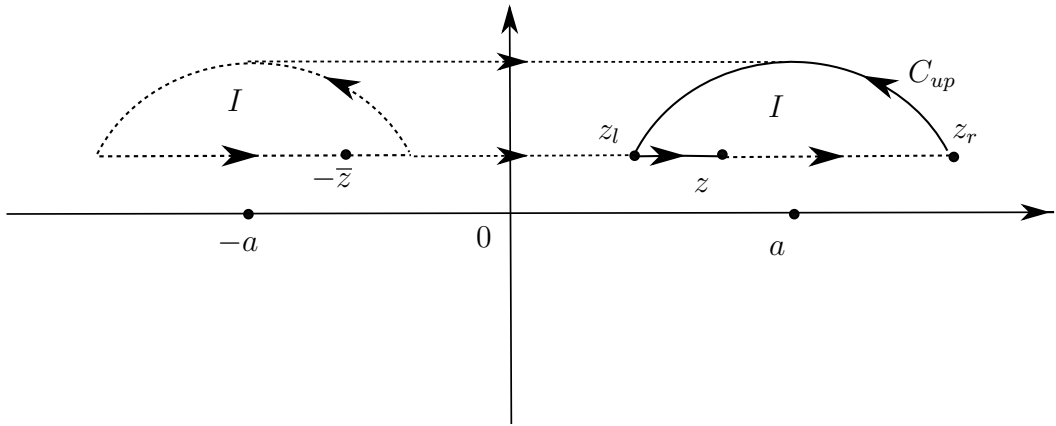
\caption{Orientation of contours and path of integration $\gamma_{1}$ (solid line) for the definition (\ref{def-phi}) of $\phi(z,k)$ on the upper sheet $\Scal_{z,1}$. The path starts at $z_{r}$, goes along the upper part $\CC_{up}$ of the circle $\CC_{a}$ to $z_{l}$ and then follows the horizontal segment up to $z$. 
%For $k$ in the different regions shown on the figure, 
The choice of the determination of the square root $\lambda$ at the initial point $z_{r}$ %is different.
of the path $\gamma_{1}$ depends on the region which contains $k$.}
\label{regionI}
\end{figure}
\begin{figure}
\centering
\def\svgwidth{14cm}
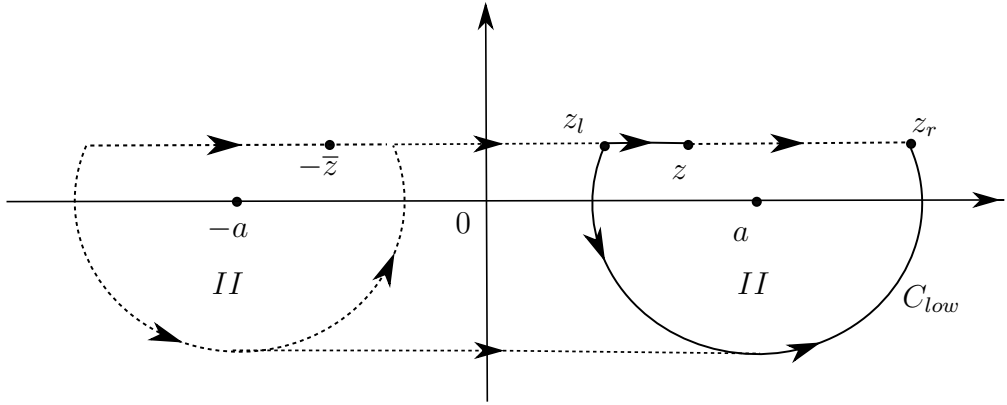
\caption{Orientation of contours and path of integration $\gamma_{2}$ (solid line) for the definition (\ref{def-phi}) of $\phi(z,k)$ on the lower sheet $\Scal_{z,2}$. The path starts at $z_{r}$, goes along the lower part $\CC_{low}$ of the circle $\CC_{a}$ to $z_{l}$ (thus describes $\CC_{low}$ clockwise) and then follows the horizontal segment up to $z$. %For $k$ in the different regions shown on the figure, 
The choice of the determination of the square root $\lambda$ at the initial point $z_{r}$ %is different.
of the path $\gamma_{2}$ depends on the region which contains $k$.}
\label{regionII}
\end{figure}
Note that, when $k$ lies outside of the convex hull of the two symmetric circles, the branch cut from $z'$ to $-\bar z'$ never intersects $k$ as $z'$ goes from $z_{r}$ to $z$ along $\gamma_{1}$ or $\gamma_{2}$. For such $k$, we may therefore use the same determination of the square root $\lambda(z',k)$ for computing the integral (\ref{def-phi}).
When $k$ lies inside one of the circles (i.e. regions I and II in Figures \ref{regionI} and \ref{regionII}), the branch cut intersects $k$ once. 
Hence, for $k\in\Scal_{z,1}$ lying in the upper part of one of the two circles, we use the determination of the square root $\lambda_{2}(z',k)$ on the first part of integration, before the crossing, and the determination $\lambda_{1}(z',k)$ after the crossing, and conversely for $k\in\Scal_{z,2}$ lying in the lower part of one of the two circles.
Finally, when $k$ lies between the two circles, the branch cut intersects $k$ twice, so that, in this case, we start integration with the determination of the square root $\lambda(z',k)$ corresponding to $k$, then change to the other determination after the first crossing, and come back to the first determination after the second crossing.

The function $\phi(z,k)$ is an analytic function of $k\in\Scal_{z}$ outside of arcs where it has jumps. Note that, in view of (\ref{defW-imp}), it has poles of order $m$ at $k\in\{z,z_{r},-\bar z,-\bar z_{r}\}$. 

As in the previous section, $\phi^{+}(z,k)$ and $\phi^{-}(z,k)$ denote the limit values of $\phi(z,k)$ when $k$ tends to the left and right of a given contour, according to its orientation. Let us compute the jumps, 
$$J(z,k)=\phi^{+}(z,k)-\phi^{-}(z,k),$$
of the function $\phi(z,k)$ on the different arcs shown in Figures \ref{regionI} and \ref{regionII}. 
As in the previous section, we denote by $\CC_{up}$ the part of the circle $\CC_{a}=C(a,1)$ above the segment $(z_{l},z_{r})$ and by $\CC_{low}$ the part of the circle $\CC_{a}$ below that segment. For $k$ on one of the two circles, we denote by $\widetilde k$ the other point on the circle with the same imaginary part, that is, we have 
\begin{align*}
\widetilde k-a & =-\overline{(k-a)},\qquad k\in \CC_{a}\\
\widetilde k+a & =-\overline{(k+a)},\qquad k\in -\CC_{a}.
\end{align*}
To compute the jumps $J(z,k)$ we first rewrites $W(z,k)$, defined in (\ref{defW-imp}), 
in terms of $ds$, the length element on $\CC_{a}$. By the same computation as the one in the previous section, we now get, with $z=x+iy\in\CC_{a}$,
%the differential form $W(z,k)$, defined in (\ref{defW}), rewrites on $C$ as
\begin{align*}%\label{Wtn}
W(z,k) & =\big((k-z)(k+\bar z)\big)^{-m-1/2}\left((k-iy)u_{t}(z)+ixu_{n}(z)\right)ds,\\
& = (k-z)^{-m-1/2}w(z,k)dz
\end{align*}
where 
$$w(z,k):=(k+\bar z)^{-m-1/2}\left((k-iy)u_{t}(z)+ixu_{n}(z)\right)\tau^{-1}(z)$$ 
and $\tau(z)$ still denotes the unit vector tangent to $\CC_{a}$ at the point $z$. 
%Let us set
%$$\tilde\partial_{t}f=\tau^{-1}(z)\partial_{t}f$$ 
%for a function $f$ on $\CC_{a}$. Because of analiticity, 
%$$\partial_{z}(k-z)^{-j-1/2}dz=\partial_{t}(k-z)^{-j-1/2}ds=\tilde\partial_{t}(k-z)^{-j-1/2}dz,\quad j=m-1,\ldots,0.$$ 
Performing $m$ integrations by parts on the integral in (\ref{def-phi}), we obtain in a similar way as in the previous section,
\begin{multline}\label{exp-phi}
\phi(z,k)=c_{0}\left[(k-z')^{-m+1/2}w(z',k)\right]_{z_{r}}^{z}+\cdots \\
+
c_{m-1}\left[(k-z')^{-1/2}\tilde\partial_{t}^{(m-1)}w(z',k)\right]_{z_{r}}^{z}
-c_{m-1}\int_{z_{r}}^{z}
(k-z')^{-1/2}\tilde\partial_{t}^{(m)}w(z',k)dz',
\end{multline}
where
%$$c_{0}=(m-1/2)^{-1},\ldots,c_{m-1}=(-1)^{m-1}\Gamma(1/2)/\Gamma(m+1/2).$$
$$c_{j}=(-1)^{j}\Gamma(m-1/2-j)/\Gamma(m+1/2),\qquad j=0,\ldots,m-1,$$
and the operator $\tilde\partial_{t}f$ is still defined by (\ref{def-der}). The bracketed terms contain the polar parts of $\phi(z,k)$, of degree $m$, at $z$ and $z_{r}$.
Note that the last integral converges when $k\in\CC_{a}$. Assume $k\in\Scal_{z,1}$ and lies on the right half of 
$\CC_{up}$. Then, from the above definition of $\phi$, we derive that
$$J(z,k)=J_{1}^{0}(z_{r},k)+\int_{z_{r}}^{k}\tilde{W}_{1}(z',k),$$
with
\begin{align}\label{def-J0}
J_{1}^{0}(z_{r},k)=2\sum_{j=0}^{m-1}
%2c_{j}(k-z_{r})^{-m+1/2}w_{1}(z_{r},k)+\cdots \\+
c_{j}(k-z_{r})^{j-m+1/2}\tilde\partial_{t}^{(j)}w_{1}(z_{r},k),\\\label{def-W1}
\tilde{W}_{1}(z',k)=2c_{m-1}
(k-z')^{-1/2}\tilde\partial_{t}^{(m)}w_{1}(z',k)dz',
\end{align}
where the subscript 1 in the above expressions means that we use the determination $\lambda_{1}$ of the square root to evaluate them. The jumps on the left half of $\CC_{up}$ on $\Scal_{z,1}$ and on $\CC_{low}$ on $\Scal_{z,2}$ can be computed in the same way. The jumps on $-\overline\CC_{up}$ and $-\overline\CC_{low}$ can be derived from the symmetry relation satisfied by $\phi(z,k)$,
$$\phi(z,-\bar k)=\overline{\phi(z,k)}.$$% it can be checked that,
%Then, from the definition of $\phi(z,k)$, 
The result is as follows. For $k\in\Scal_{z,1}$, 
\allowdisplaybreaks{
\begin{alignat}{2}
J(z,k) & =J_{1}^{0}(z_{r},k)+\int_{(z_{r},k)}\tilde{W}_{1}(z',k),\quad  & & k\text{ on the right half of }\CC_{up},\label{D11}\\
J(z,k) & =J_{1}^{0}(z_{r},k)+\int_{(z_{r},\widetilde k)}\tilde W_{1}(z',k)-\int_{(\widetilde k,k)}\tilde W_{1}(z',k),\quad  & & k\text{ on the left half of }\CC_{up},\label{D12}\\
J(z,k) & = \overline{J(z,-\bar k)},\quad & & k\text{ on }-\overline \CC_{up},\label{D13}
\end{alignat}
}
where the path of integration $(a,b)$ in each of the above integrals is the subarc of $\CC_{up}$ from $a$ to $b$, in the positive direction.
% and where in $W(z',k)$ we use the determination $\lambda_1$ of the square root. 
For $k\in\Scal_{z,2}$, we have
\allowdisplaybreaks{
\begin{alignat}{2}
J(z,k) & =J_{2}^{0}(z_{r},k)+\int_{(z_{r},k)}\tilde{W}_{2}(z',k),
%-2\int_{(z_{r},k)}W_{2}(z',k),
\quad  & & k\text{ on the right half of }\CC_{low},
\label{D21}\\
J(z,k) & =J_{2}^{0}(z_{r},k)+\int_{(z_{r},\widetilde k)}\tilde W_{2}(z',k)-\int_{(\widetilde k,k)}\tilde W_{2}(z',k),
%-2\int_{(z_{r},\widetilde k)}W_{2}(z',k)+2\int_{(\widetilde k,k)}W_{2}(z',k),
\quad  & & k\text{ on the left half of }\CC_{low},\label{D22}\\
J(z,k) & =\overline{J(z,-\bar k)},\quad & & k\text{ on }-\overline \CC_{low},\label{D23}
\end{alignat}
}
where $(a,b)$ in each of the integrals denotes the subarc of $\CC_{low}$ from $a$ to $b$, in the negative direction. %and where in $W(z',k)$ we use the determination $\lambda_2$ of the square root.

Note that the above jumps take place on two open contours on the Riemann surface $\Scal_{z}$. The first one starts at $z_{r,1}$ on the first sheet, follows the circle $\CC_{a}$, goes through the cut at $z_{l}$ and finishes at $z_{r,2}$ on the second sheet. The second  one start at $-\bar z_{r,1}$ on the first sheet, follows the circle $-\CC_{a}$, goes through the cut at $-\bar z_{l}$ and finishes at $-\bar z_{r,2}$ on the second sheet. Note also that the jump on the first contour is continuous at $z_{l}$ (compare (\ref{D12}) and (\ref{D22}) at $k=z_{l}$), and, by symmetry, the same is true of the jump on the second contour at $-\bar z_{l}$.
% (compare (\ref{D14}) and (\ref{D24}) at $k=-\bar z_{l}$).

It follows, simply from the definition (\ref{def-phi}) of $\phi(z,k)$ and the choice of the determinations, that, on each sheet, $\phi(z,k)$ has no jumps on $(z,z_{r})$ and $(-\bar z_{r},-\bar z)$. On the branch cut
%We let the reader check that, on each sheet, $\phi(z,k)$ has no other jumps, and also no jump across the branch cut 
$(-\bar z,z)$ of the Riemann surface $\Scal_z$ there is also no jumps. Indeed, on  the part $(-\bar z_{l},z_{l})$, $k$ is not close from the paths of integration and we can use the definition (\ref{def-phi}) to compute the jumps. One has
$$\phi^{+}_{1}(z,k)-\phi^{-}_{2}(z,k)
=\int_{z_{r}}^{z}W^{+}_{1}(z',k)-\int_{z_{r}}^{z}W^{-}_{2}(z',k)=0,$$
$$\phi^{-}_{1}(z,k)-\phi^{+}_{2}(z,k)
=\int_{z_{r}}^{z}W^{-}_{1}(z',k)-\int_{z_{r}}^{z}W^{+}_{2}(z',k)=0,$$
where we have deformed the original paths of integration $\gamma_{1}$ and $\gamma_{2}$ into the segment $(z_{r},z)$, which is possible since the differential form $W(z,k)$ is closed. 
When $k\in\Scal_{z,1}$ is on the $-$ side of $(z_{l},z)$ or $k\in\Scal_{z,2}$ is on the $+$ side of $(z_{l},z)$, we may deform both paths $\gamma_{1}$ and $\gamma_{2}$ into the segment $(z_{r},z)$ and we get
$$\phi^{-}_{1}(z,k)-\phi^{+}_{2}(z,k)=\int_{z_{r}}^{z}W^{-}_{1}(z',k)-\int_{z_{r}}^{z}
W^{+}_{2}(z',k)=0.$$
Finally, when $k\in\Scal_{z,1}$ is on the $+$ side of $(z_{l},z)$ or $k\in\Scal_{z,2}$ is on the $-$ side of $(z_{l},z)$, we deform both paths $\gamma_{1}$ and $\gamma_{2}$ into the segment $(z_{r},k)$ followed by the segment $(k,z)$. 
Here $k$ is closed from the paths of integration, so we use (\ref{exp-phi}) to compute the jumps. We get
\begin{align*}
\phi^{+}_{1}(z,k) & =-\frac12\int_{z_{r}}^{k}\tilde W^{-}_{1}(z',k)-\frac12\int_{k}^{z}\tilde W^{+}_{1}(z',k)\\
& =-\frac12\int_{z_{r}}^{k}\tilde W^{+}_{2}(z',k)-\frac12\int_{k}^{z}\tilde W^{-}_{2}(z',k)=\phi^{-}_{2}(z,k),
\end{align*}
where we note that the bracketed terms in (\ref{exp-phi}) do not contribute to the jumps.

For completeness, let us remark that, if we would consider $\phi(z,k)$ as a function on one of the two sheets only, e.g. $\Scal_{z,1}$, then it would have a jump on the segment 
$(-\bar z_{l},z_{l})$, considered as an arc on $\Scal_{z,1}$, namely
$$\int_{\CC_{up}}\tilde W_{1}(z',k),\quad k\in (-\bar z_{l},z_{l}).$$
On the second sheet $\Scal_{z,2}$, a jump would also occurs,
$$-\int_{\CC_{low}}\tilde W_{2}(z',k),\quad k\in (-\bar z_{l},z_{l}),$$
which is actually opposite to the previous one. 

The term $J^{0}(z_{r},k)$ has a polar singularity of order $m$ at $k=z_{r}$ so that the jumps (\ref{D11})--(\ref{D23}) have polar singularities of order $m$ at either $z_{r}$ or $-\bar z_{r}$. Hence, instead of $\phi(z,k)$, we consider
$$\tilde{\phi}(z,k)=\big((k-z_{r})(k+\bar{z}_{r})\big)^{m}\phi(z,k),$$
whose jumps 
\begin{equation}\label{tilde-jump}
\tilde J(z,k)=\big((k-z_{r})(k+\bar{z}_{r})\big)^{m}J(z,k)
\end{equation}
are regular (and lie on the same contours as those of $\phi(z,k)$). 
%From the values (\ref{D11})--(\ref{D24}) of the jumps $J(z,k)$ and the symmetry relation (\ref{sym-W}) for $W$ follows a similar symmetry for $\tilde J(z,k)$, namely
%\begin{equation}\label{sym-D}
%\tilde J(z,-\overline k)=\overline{\tilde J(z,k)}.
%\end{equation}
From the definitions of $W$ and $\phi$, and in view of (\ref{inf1})--(\ref{inf2}), we have
\begin{align}
\lim_{k\to\infty_{1}}\tilde\phi(z,k)& =\int_{z_{r}}^{z}du=u(z)-u(z_{r}),\label{beh-phi-1}\\
\lim_{k\to\infty_{2}}\tilde\phi(z,k)& =-\int_{z_{r}}^{z}du=u(z_{r})-u(z),\label{beh-phi-2}
\end{align}
so that, in particular,
\begin{equation}\label{beh-phi-inf}
\tilde\phi(z,\infty_{1})=-\tilde\phi(z,\infty_{2}).
\end{equation}
Next, the function $k\to\tilde\phi(z,k)$ remains bounded near the four endpoints $z_{r,1}$, $z_{r,2}$, $-\bar{z}_{r,1}$, $-\bar{z}_{r,2}$ of the two contours where the jumps
(\ref{D11})--(\ref{D23}) take place. Indeed, near $z_{r,1}$, $W(z,k)$ is of order $(k-z_{r,1})^{-m-1/2}$ and consequently $\tilde\phi(z,k)$ is of order $(k-z_{r,1})^{1/2}$. The same fact holds true near the three other points. As a last remark, let us mention that 
$\tilde\phi(z,k)$ has poles of order $m$ at $z$ and $-\bar z$ since this holds true for $\phi(z,k)$. We denote by $\tilde\phi_{z,-\bar z}(z,k)$ the sum of its polar parts at $z$ and $-\bar z$.

The jumps (\ref{tilde-jump}), the relation (\ref{beh-phi-inf}) between the values at infinities and the boundedness near the endpoints completely characterize the function $\tilde\phi-\tilde\phi_{z,-\bar z}$ on $\Scal_{z}$. Indeed, if there are two such functions, then their difference would be analytic on the compact Riemann surface $\overline\Scal_{z}$, hence constant. Since it would also satisfy the relation (\ref{beh-phi-inf}), it could only be the zero function.  An explicit expression can be given for the unique solution of the Riemann-Hilbert problem defined by the previous conditions (jumps, relation between the values at $\infty_{1}$ and $\infty_{2}$, and boundedness near the endpoints), namely
\begin{multline}\label{plemelj}
\tilde\phi(z,k)-\tilde\phi_{z,-\bar z}(z,k)=\frac{1}{4i\pi}\int_{\CC_{up}\cup-\overline{\CC}_{up}}\tilde J(z,k')
\left(\frac{\lambda(z,k)%(k+\bar z)
}{\lambda_{1}(z,k')%(k'+\bar z)
}+1\right)\frac{dk'}{k'-k}
\\[10pt]
+\frac{1}{4i\pi}\int_{\CC_{low}\cup-\overline{\CC}_{low}}\tilde J(z,k')
\left(\frac{\lambda(z,k)%(k+\bar z)
}{\lambda_{2}(z,k')%(k'+\bar z)
}+1\right)\frac{dk'}{k'-k},
\end{multline}
where the contours of integration are oriented as in Figures \ref{regionI} and \ref{regionII}. In the first integral $k'\in\Scal_{z,1}$ and in the second integral $k'\in\Scal_{z,2}$. Let us check that
the expression in (\ref{plemelj}) satisfies the characterizing properties of $\tilde\phi-\tilde\phi_{z,-\bar z}$. Indeed, it defines an analytic function of $k$ on $\Scal_{z}$ outside of the two contours $\CC_{up}\cup \CC_{low}$ and 
$-\overline{\CC}_{up}\cup-\overline{\CC}_{low}$.
Thanks to the Plemelj formula, we see that it has the right jumps on these contours. Because of (\ref{inf1})--(\ref{inf2}), it also satisfies relation (\ref{beh-phi-inf}). From the formulas (\ref{D11}), (\ref{D13}), (\ref{D21}), (\ref{D23}), we get that $\tilde J(z,k')$ vanishes at the endpoints of the two contours.
Hence, the expression in (\ref{plemelj})
remains bounded near these endpoints, see \cite[Lemma 7.2.2]{AF} for details. 

Making use of (\ref{beh-phi-1}), we deduce from (\ref{plemelj}), where we rename $k'$ as $k$, that
\begin{align*}
u(z)-u(z_{r})
& =-\frac{1}{4i\pi}\int_{\CC_{up}\cup-\overline{\CC}_{up}}\frac{\tilde J(z,k)}{
\lambda_{1}(z,k)}dk
-\frac{1}{4i\pi}\int_{\CC_{low}\cup-\overline{\CC}_{low}}\frac{\tilde J(z,k)}{
\lambda_{2}(z,k)}dk
\\
 & =-\frac{1}{2\pi}\Im\int_{\CC_{up}}\frac{\tilde J(z,k)}{
\lambda_{1}(z,k)}dk
-\frac{1}{2\pi}\Im\int_{\CC_{low}}\frac{\tilde J(z,k)}{
\lambda_{2}(z,k)}dk\\
& =-\frac{1}{2\pi}\Im\int_{\CC_{a}}\frac{\tilde J(z,k)}{
\sqrt{(k-z)(k+\bar z)}}dk.
\end{align*}
Note that the polar part $\tilde\phi_{z,-\bar z}$ does not give any contribution in the above computation as $k$ tends to infinity.
In the second equality we have used the first identity in (\ref{sym-W}) and the fact that
$$\tilde J(z,-\bar k)=\overline{\tilde J(z,k)}.$$ 
In the last expression, we integrate along the circle $\CC_{a}$ oriented counter-clockwise, starting at $z_{r}$ with the determination of the square root that behaves like $k$ at infinity. At $k=z_{l}$, the determination changes so that the square root remains continuous along the path of integration. This finishes the proof of Theorem \ref{expl-odd}.
\end{proof}
%\subsection{Case of a general Jordan domain $\Omega$}\label{gen-domain}
The method used in the proofs of Theorems \ref{expl-even} and \ref{expl-odd} could be applied in case of a general, bounded, simply-connected domain $\Omega$ in the right half-plane $\H$. For the definition of the corresponding function $\phi(z,k)$, in particular the paths of integration $\gamma_{1}$ and $\gamma_{2}$, one could replace the points $z_{r}$ (resp. $z_{l}$) by e.g. a point $\omega_{r}$ (resp. $\omega_{l}$) on the boundary of $\Omega$ of smallest (resp. largest) abscissa (hence independant of $z$). The segment $(z_{l},z)$ may be replaced with any path from $\omega_{l}$ to $z$. In case of an even coefficient $\alpha$, another path from $\omega_{r}$ to $z$ can be fixed to obtain a path from $\omega_{r}$ to $\omega_{l}$ that separates $\Omega$ into a lower and an upper part.
\section{The Dirichlet-Neumann map for the case of an even coefficient $\alpha$}
\label{global}
%Here we consider the case $\alpha=-2(m-1)$, $m\in\N^{*}$. 
This section is devoted to the proof of Theorem \ref{corresp}. Theorems \ref{expl-even} and \ref{expl-odd} give an expression of the solution to (\ref{Grad-S}), and in particular of the jump $J(z,k)$, in terms of its tangential and normal derivatives on the boundary of the disk $\DD_{a}$. In practice, only one type of boundary data is usually known and it is thus important to determine if this information is sufficient for computing the solution.

Using the link (\ref{polar}) between complex and directional derivatives, the global relation, derived from Proposition \ref{Prop-W} and Poincar\'e lemma, 
\begin{equation*}%\label{rel-glob}
\int_{\CC_{a}}[(k-z)(k+\bar z)]^{\alpha/2-1}
((k+\bar z)u_{z}(z)dz+(k-z)u_{\bar z}(z)d\bar z)=0,%\qquad k\in\C\setminus(\DD_{a}\cup\DD_{-a}),
\end{equation*}
can be rewritten as
\begin{equation}\label{rel-glob}
\int_{\CC_{a}}[(k-z)(k+\bar z)]^{\alpha/2-1}
((y+ik)u_{t}(z)-xu_{n}(z))ds=0,%\qquad k\in\C\setminus(\DD_{a}\cup\DD_{-a}),
\end{equation}
for $k\in\C\setminus(\DD_{a}\cup\DD_{-a})$. We first show that, when $\alpha$ is an even negative integer, (\ref{rel-glob})
allows one to recover the Neumann data $u_{n}$ from the Dirichlet data $u_{t}$.
%, that is we prove assertion (i) of Theorem \ref{corresp}. 
%We first show that $u_{t}=0$ implies $u_{n}=0$, namely we prove the following proposition.
\subsection{Case of an even negative integer coefficient $\alpha$}
\begin{proof}[Proof of Theorem \ref{corresp} when $\alpha=-2(m-1)$, $m\in\N^{*}$]
Equivalently to recovering the Neumann data $u_{n}$, we may recover the function 
\begin{equation}\label{def-f}
f(z):=(x+a)u_{n}(z+a),\qquad z\in\T,
\end{equation}
from the relation
\begin{equation}\label{rel-glob2}
\int_{\T}\frac{z^{m-1}f(z)}{(z-(k-a))^{m}(z+(k+a)^{-1})^{m}}dz=
\int_{\T}\frac{z^{m-1}(y+ik)u_{t}(z+a)}{(z-(k-a))^{m}(z+(k+a)^{-1})^{m}}dz,
\end{equation}
where  $k\in\C\setminus(\DD_{a}\cup\DD_{-a})$ and we integrate on $\T$ instead of $\CC_{a}$. 
%Here $f\in L^{2}(\T)$ so that the above equality actually holds almost everywhere on $\T$.

Now, we set 
$$
%\begin{lemma}
%Assume relation (\ref{eq-inject}) holds true.
%Then the function $h(z)$ satisfies a differential equation of order $m-1$,
%\begin{equation}\label{diff-eq}
%\sum_{p=0}^{m-1}\alpha_{p}z^{m-1-p}(z^{2}+2az+1)^{p}h^{(p)}(z)=P_{2m-2}(z),\quad \alpha_{p}=(-1)^{m-1+p}\frac{(2m-p-2)!}{(m-p-1)!p!},
%\end{equation}
%where $P_{2m-2}(z)$ is a polynomial of degree at most $2m-2$.
%\end{lemma}
\mu:=-1/(k+a),\qquad \varphi(z):=\frac{-z}{1+2az}.$$
The map $\varphi$ is involutive and sends $\mu$ to $1/(k-a)$. Moreover, 
$k\in\C\setminus(\DD_{a}\cup\DD_{-a})$ is equivalent to the fact that $\mu$ (or $\varphi(\mu)$) belongs to the annulus 
$$\A:=\D\setminus D\left(\frac{-2a}{4a^{2}-1},\frac{1}{4a^{2}-1}\right),$$
so that (\ref{rel-glob2}) can be rewritten
\begin{equation}\label{rel-glob3}
\int_{\T}\frac{z^{m-1}f(z)}{(1-\varphi(\mu)z)^{m}(z-\mu)^{m}}dz=
\frac1\mu\int_{\T}\frac{z^{m-1}(\mu y-i(a\mu+1))u_{t}(z+a)}{(1-\varphi(\mu)z)^{m}(z-\mu)^{m}}dz,
\qquad \mu\in\A.
\end{equation}
%where
%\begin{equation}\label{def-psi}
%\qquad\Psi(\mu):=\frac1\mu\int_{\T}\frac{z^{m-1}(\mu y-i(a\mu+1))u_{t}(z+a)}{(1-\varphi(\mu)z)^{m}(z-\mu)^{m}}dz.
%\end{equation}
The functions $f(z), u_{t}(z+a), yu_{t}(z+a)$ are real-valued in $L^{2}(\T)$, so they can be decomposed as 
\begin{equation}\label{f-g}
f(z)=g_{1}(z)+\bar g_{1}(1/z),\quad u_{t}(z+a)=g_{2}(z)+\bar g_{2}(1/z),\quad
yu_{t}(z+a)=g_{3}(z)+\bar g_{3}(1/z),
\end{equation}
with $g_{i}(z)$ holomorphic on $\D$ and  $\bar g_{i}(1/z)$ holomorphic on $\bar\C\setminus\D$ for $i=1,2,3$. Note that the imaginary parts of $g_{1}(0), g_{2}(0), g_{3}(0)$ can be chosen at will in the decompositions (\ref{f-g}). 
%In the sequel, we will assume that $\Im g(0)=0$. 
For $i=1,2,3$, let $h_{i}(z):=z^{m-1}g_{i}(z)$, and
\begin{equation}\label{def-Phi}
\Phi_{i}(\mu):=\int_{\T}\frac{z^{m-1}g_{i}(z)}{(1-\varphi(\mu)z)^{m}(z-\mu)^{m}}dz
=\frac{2i\pi}{(m-1)!}\left(\frac{h_{i}(z)}{(1-\varphi(\mu)z)^{m}}
\right)^{(m-1)}(\mu),\qquad\mu\in\A.
\end{equation}
It is clear that the functions $\Phi_{i}$ are analytic in $\A$. Since 
$$1-\varphi(\mu)\mu=(\mu^{2}+2a\mu+1)/(1+2a\mu),$$ 
one derives from the second expressions in (\ref{def-Phi}) that the $\Phi_{i}$ extend analytically to $\D$, except at the point $z_{1}$, where we denote by $z_{1}$ and $z_{2}$ the two roots of $z^{2}+2az+1=0$,
$$z_{1}=-a+\sqrt{a^{2}-1}\in D,\qquad z_{2}=\varphi(z_{1})=-a-\sqrt{a^{2}-1}\in\C\setminus\D.$$
At the point $z_{1}$, the functions $\Phi_{i}$ have a polar singularity of order at most $2m-1$. We also note that the $\Phi_{i}$ have a zero of order at least $m$ at $-1/2a$.

Next, from the fact that $\varphi$ is involutive, follows that, for any function $g(z)$ on $\T$,
$$
\int_{\T}\frac{z^{m-1}\bar g(1/z)}{(1-\varphi(\mu)z)^{m}(z-\mu)^{m}}dz
=-\int_{\T}\bar{\frac{z^{m-1}g(z)}{(1-\bar\mu z)^{m}(z-\varphi(\bar\mu))^{m}}dz}.
$$
Therefore, the relation (\ref{rel-glob3}) can be rewritten as 
\begin{equation}\label{rel-glob4}
\Phi_{1}(\mu)-\bar{\Phi}_{1}(\varphi(\mu))=\Phi_{3}(\mu)-\bar{\Phi}_{3}(\varphi(\mu))
-i(a+1/\mu)(\Phi_{2}(\mu)-\bar{\Phi}_{2}(\varphi(\mu))),\qquad\mu\in\A.
\end{equation}
%The function $\varphi$ maps $\A$ onto itself and sends the inner disk $D$ onto $\C\setminus\D$ with the inner circle of $\A$ onto the unit circle $\T$. 
%Evaluating the above relation at $\varphi(\bar\mu)$, and taking conjugates leads to
%\begin{equation*}%\label{rel-glob4}
%\bar{\Phi}_{1}(\varphi(\mu))-\Phi_{1}(\mu)=\bar{\Phi}_{3}(\varphi(\mu))-\Phi_{3}(\mu)
%+i(a+1/\varphi(\mu))(\bar{\Phi}_{2}(\varphi(\mu))-\Phi_{2}(\mu)),\qquad\mu\in\A.
%\end{equation*}
%Substracting the second equation from the first one, and dividing by 2, we obtain
%\begin{equation*}%\label{rel-glob4}
%\Phi_{1}(\mu)-\bar{\Phi}_{1}(\varphi(\mu))=\Phi_{3}(\mu)-\bar{\Phi}_{3}(\varphi(\mu))
%-(i/2)(1/\mu-1/\varphi(\mu))(\Phi_{2}(\mu)-\bar{\Phi}_{2}(\varphi(\mu))),\qquad\mu\in\A.
%\end{equation*}
%Next, we observe that
%$$\frac12\left(\frac1\mu-\frac{1}{\varphi(\mu)}\right)=\frac{1+a\mu}{\mu}=
%-\frac{1+a\varphi(\mu)}{\varphi(\mu)},$$
%so that the last equation may be rewritten as
%\begin{equation}\label{rel-glob5}
%\Phi_{1}(\mu)-\bar{\Phi}_{1}(\varphi(\mu))=\Phi_{3}(\mu)-\bar{\Phi}_{3}(\varphi(\mu))
%-a(i\Phi_{2})(\mu)+a(\bar{i\Phi_{2}})(\varphi(\mu))-\Psi_{2}(\mu)+\bar{\Psi}_{2}(\varphi(\mu)),
%\end{equation}
%for $\mu\in\A$, where we have set
%$$\Psi_{2}(\mu):=i\Phi_{2}(\mu)/\mu.$$
Note that $g_{2}(0)=0$ since the Fourier coefficient of $u_{t}(z+a)$ of order 0 vanishes. In view of (\ref{def-Phi}), this entails that $\Phi_{2}(0)=0$ so that the function $\Phi_{2}(\mu)/\mu$ is analytic at 0. 
% EXPRESSION INTEG DE PSI_{2}
%It also has expressions similar to those in (\ref{def-Phi}). Indeed, with 
%$${g}_{2}(z) :=z\tilde{g}_{2}(z),\qquad z^{m}=(z-\mu)^{m}+\mu Q_{m}(z,\mu),$$ 
%where $Q_{m}(z,\mu)$ is a polynomial of degree $m-1$ in each variable, we have, for $\mu\in\D$,
%\begin{align}\notag
%\Psi_{2}(\mu) & =\frac{i}{\mu}\int_{\T}\frac{z^{m}\tilde{g}_{2}(z)}{(1-\varphi(\mu)z)^{m}(z-\mu)^{m}}dz
%=i\int_{\T}\frac{Q_{m}(z,\mu)\tilde{g}_{2}(z)}{(1-\varphi(\mu)z)^{m}(z-\mu)^{m}}dz
%\\ \label{exp-psi2}
%& =\frac{-2\pi}{(m-1)!}\left(\frac{Q_{m}(z,\mu)\tilde{g}_{2}(z)}{(1-\varphi(\mu)z)^{m}}
%\right)^{(m-1)}(\mu).
%\end{align}
%%%%%%%%%%%
%Denoting
%$$\Phi(\mu):=\Phi_{1}(\mu)-\Phi_{3}(\mu)+a(i\Phi_{2})(\mu)+\Psi_{2}(\mu),$$
%equation (\ref{rel-glob5}) becomes
%\begin{equation}\label{rel-glob6}
%\Phi(\mu)-\bar{\Phi}(\varphi(\mu))=0,\qquad\mu\in\A,
%\end{equation}
%where $\Phi(\mu)$ is analytic in $\D$, except at $z_{1}$ where it has a pole of order  at most $2m-1$. It also has a zero of order at least $m$ at $-1/2a$.
%%%%%%%%%%%%%%%%%
Multiplyimg both sides of the equation by 
$$S(\mu)=(\mu-z_{1})^{2m-1}(\mu-z_{2})^{2m-1}/(2a\mu+1)^{m},$$
and rearranging terms, we get
\begin{multline*}
S(\mu)(\Phi_{1}(\mu)-\Phi_{3}(\mu)+i(a+1/\mu)\Phi_{2}(\mu))\\=
S(\mu)(\bar{\Phi}_{1}(\varphi(\mu))-\bar{\Phi}_{3}(\varphi(\mu))
+i(a+1/\mu)\bar{\Phi}_{2}(\varphi(\mu))).
\end{multline*}
The function on the left-hand side is analytic in $\D$. Moreover,
the function $\varphi$ maps $\A$ onto itself and sends the inner disk $D$ onto $\C\setminus \D$ with the inner circle of $\A$ onto the unit circle $\T$. Also, $-1/2a$ and $\infty$ map each others by $\varphi$.
Hence, the term on the right-hand side is analytic in $\C\setminus\D$ and has a polar singularity of order $2m-2$ at infinity. Equating the left-hand side to the part in the right-hand side that is analytic in $\D$, we get
$$S(\mu)\Phi_{1}(\mu)=S(\mu)\Psi(\mu)+P_{2m-2}(\mu),
\qquad\Psi(\mu):=\Phi_{3}(\mu)-i(a+1/\mu)\Phi_{2}(\mu),$$
where $\Psi(\mu)$ is a known function depending on $u_{t}$, and $P_{2m-2}(\mu)$ is a polynomial of degree at most $2m-2$, that depends on both $u_{t}$ and $u_{n}$, (hence is unknown).
%Hence, the relation (\ref{rel-glob6}) allows one to continue $\Phi$ analytically to $\bar{\C}\setminus\{z_{1},z_{2}\}$,
% and to derive that it is a rational function of the form
%\begin{equation}\label{Phi-rat}
%\Phi(\mu)=\frac{2i\pi}{(m-1)!}\frac{(2a\mu+1)^{m}P_{2m-2}(\mu)}{(\mu-z_{1})^{2m-1}(\mu-z_{2})^{2m-1}},
%\end{equation}
%where $P_{2m-2}(\mu)$ is some polynomial of degree at most $2m-2$ (note that 
Making use of the second expressions in (\ref{def-Phi}) of the $\Phi_{i}$ functions and applying Leibniz rule on the $(m-1)$-th derivatives, we readily obtain that $h_{1}(z)$ satisfies the differential equation of order $m-1$,
\begin{equation}\label{diff-eq}
\sum_{p=0}^{m-1}\alpha_{p}z^{m-1-p}(z^{2}+2az+1)^{p}h_{1}^{(p)}(z)
=H(z)+P_{2m-2}(z),%\quad \alpha_{p}=(-1)^{m-1+p}\frac{(2m-p-2)!}{(m-p-1)!p!},
\end{equation}
where
$$
\alpha_{p}=(-1)^{m-1+p}\frac{(2m-p-2)!}{(m-p-1)!p!},\qquad p=0,\ldots,m-1,
$$
and $H(z)$ is a known function, analytic in $\D$, that is given in terms of $h_{2}(z)$ and $h_{3}(z)$ by
$$H(z)=\sum_{p=0}^{m-1}\alpha_{p}z^{m-1-p}(z^{2}+2az+1)^{p}
(h_{3}^{(p)}(z)-i(a+1/z)h_{2}^{(p)}(z)).$$
%\begin{lemma}
%Assume that the global relation (\ref{rel-glob2}) holds true.
%Then the functions
%%$$h(z)=h_{1}(z)-h_{3}(z)+iah_{2}(z)+iQ_{m}(z,\mu)\tilde{g}_{2}(z),$$ 
%$$h(z)=h_{1}(z)-h_{3}(z)+iah_{2}(z),$$ 
%and $h_{2}(z)$ satisfy a differential relation%equation of order $m-1$,
%\begin{equation}\label{diff-eq}
%\sum_{p=0}^{m-1}\alpha_{p}z^{m-1-p}(z^{2}+2az+1)^{p}(zh^{(p)}(z)+ih_{2}^{(p)}(z))
%=zP_{2m-2}(z),\quad \alpha_{p}=(-1)^{m-1+p}\frac{(2m-p-2)!}{(m-p-1)!p!},
%\end{equation}
%where $P_{2m-2}(z)$ is a polynomial of degree at most $2m-2$.
%\end{lemma}
%\begin{proof}
%%function $h$ satisfies the differential equation
%\end{proof}
By Lemma \ref{eq-diff-pol} which, for convenience, we have stated and proved after the present proof, there exists a unique solution $\tilde P_{2m-2}(z)$, polynomial of degree at most $2m-2$, to the differential equation
$$
\sum_{p=0}^{m-1}\alpha_{p}z^{m-1-p}(z^{2}+2az+1)^{p}\tilde P_{2m-2}^{(p)}(z)
=P_{2m-2}(z).
$$
Hence, $h_{1}(z)-\tilde P_{2m-2}(z)$ coincides with the solution $h(z)$, analytic at $z_{1}$, of the differential equation
$$
\sum_{p=0}^{m-1}\alpha_{p}z^{m-1-p}(z^{2}+2az+1)^{p}h^{(p)}(z)=H(z).
$$
%that is analytic in a neighborhood of $z_{1}$ : EXISTENCE ET CONSTRUCTION A JUSTIFIER\\
% we obtain that
%$$h_{1}(z)=\tilde h(z)+\tilde P_{2m-2}(z),$$
%where 
Note that the function $h(z)$ can be explicitely computed as a series expansion around $z_{1}$, see Lemma \ref{eq-diff-pol}. We know that $h_{1}(z)$ vanishes at the origin with order $m-1$, hence it can be written as
$$h_{1}(z)=z^{m-1}(\tilde h(z)+\tilde P_{m-1}(z)),$$
where $\tilde h(z)$ is known and $\tilde P_{m-1}(z)$ is some polynomial of degree at most $m-1$. From the link between $f(z)$ and $h_{1}(z)$ follows that
$$z^{m-1}f(z)=z^{m-1}(\tilde h(z)+\bar{\tilde h}(1/z))+Q_{2m-2}(z),$$
where 
$$Q_{2m-2}(z)=z^{m-1}(\tilde P_{m-1}(z)+\bar{\tilde P}_{m-1}(1/z))$$ 
is a polynomial of degree at most $2m-2$. From the global relation (\ref{rel-glob3}) and the fact that $\tilde h(z)$ is known, 
%and with $Q_{3m-3}(z)=z^{m-1}Q_{2m-2}(z)$, 
we derive that the following integral is also known for $\mu\in\A$,
\begin{align*}
\frac{S(\mu)}{2i\pi}\int_{\T}\frac{Q_{2m-2}(z)}{(1-\varphi(\mu)z)^{m}(z-\mu)^{m}}dz
& =\frac{S(\mu)}{(m-1)!}\left(\frac{Q_{2m-2}(z)}{(1-\varphi(\mu)z)^{m}}
\right)^{(m-1)}(\mu)
\\
& =\frac{1}{(m-1)!}\sum_{p=0}^{m-1}\alpha_{p}\mu^{m-1-p}(\mu^{2}+2a\mu+1)^{p}
Q_{2m-2}^{(p)}(\mu)
\end{align*}
Consequently, $Q_{2m-2}(z)$ solves a differential equation of the type (\ref{diff-eq-H}) with a second member which is necessarily a polynomial, hence analytic at $z_{1}$. From Lemma \ref{eq-diff-pol}, we know that $Q_{2m-2}$ is uniquely determined and can be recovered from this equation. Therefore, $\tilde P_{m-1}(z)$, $h_{1}(z)$, $g_{1}(z)$, and finally $f(z)$ are also uniquely recovered.
%Since the functions $h_{i}(z)$, $i=1,2,3$, are all analytic at $z_{1}$, equation (\ref{diff-eq}) entails that
%$$h_{1}^{(p)}(z_{1})=h_{3}^{(p)}(z_{1})-i(a+1/z_{1})h_{2}^{(p)}(z_{1}),\qquad p=0,\ldots,m-1.$$
%Consequently, the function $h_{1}(z)$ is uniquely determined by the differential equation (\ref{diff-eq}).
%
%We are now in a position to finish the proof of Theorem \ref{corresp}.
%%\begin{proof}[Proof of Proposition \ref{inject}]
%From the previous lemma, we know that $g$ is a polynomial of degree at most $m-1$. Hence
%$$F(z):=z^{m-1}f(z)=z^{m-1}g(z)+z^{m-1}\bar g(1/z)$$
%is a polynomial of degree at most $2m-2$. Also, from (\ref{rel-glob3}),
%\begin{equation*}
%\int_{\T}\frac{F(z)}{(1-\varphi(\mu)z)^{m}(z-\mu)^{m}}dz
%=\left(\frac{F(z)}{(1-\varphi(\mu)z)^{m}}\right)^{(m-1)}(\mu)=0,\qquad \mu\in\A.
%\end{equation*}
%Applying Leibniz rule to evaluate the previous derivative, we obtain that
%$$
%\sum_{p=0}^{m-1}\alpha_{p}\frac{\mu^{m-1-p}}{(\mu^{2}+2a\mu+1)^{m-1-p}}F^{(p)}(\mu)=0,\qquad \mu\in\A.$$
%The expression in the left-hand side is a rational function, hence it vanishes everywhere, which entails in particular that
%$$F^{(p)}(z_{1})=F^{(p)}(z_{2})=0,\qquad p=0,\ldots,m-1.$$
%Since $F$ is a polynomial of degree at most $2m-2$, it is identically 0 and the same holds true for the function $f$ on $\T$ and the function $u_{n}$ on $\CC_{a}$.
\end{proof}
\begin{lemma}\label{eq-diff-pol}
Let $H(z)$ be a function analytic in a neighborhood of $z_{1}$. There exists at most one solution $h(z)$ of the differential equation
\begin{equation}\label{diff-eq-H}
\sum_{p=0}^{m-1}\alpha_{p}z^{m-1-p}(z^{2}+2az+1)^{p}
h^{(p)}(z)
=H(z),
%\quad \alpha_{p}=(-1)^{m-1+p}\frac{(2m-p-2)!}{(m-p-1)!p!},
\end{equation}
which is analytic in a neighborhood of $z_{1}$.
Moreover, if $H(z)=P_{2m-2}(z)$ is a polynomial of degree at most $2m-2$, then
%\begin{equation}\label{diff-eq-0}
%\sum_{p=0}^{m-1}\alpha_{p}z^{m-1-p}(z^{2}+2az+1)^{p}
%h^{(p)}(z)
%=P_{2m-2}(z),
%%\quad \alpha_{p}=(-1)^{m-1+p}\frac{(2m-p-2)!}{(m-p-1)!p!},
%\end{equation}
% that is analytic in a neighborhood of $z_{1}$.
the solution $h(z)$ exists and is also a polynomial of degree at most $2m-2$. 
\end{lemma}
\begin{proof}
The two roots $z_{1}$ and $z_{2}$ of $z^{2}+2az+1$ are regular singular points of the differential equation (\ref{diff-eq-H}). We rewrite it in a neighborhood of $z_1=-a+\sqrt{a^2-1}\in D$, leading to
$$
\sum_{p=0}^{m-1} \alpha_p (z+z_1)^{m-1-p} (z+\beta)^p z^p \tilde{h}^{(p)}(z)
=H(z+z_1),
$$
where $\beta=2\sqrt{a^2-1}$ and $\tilde{h}(z)=h(z+z_1)$. Denoting by $\displaystyle \sum_{k=0}^{\infty} a_k z^k$ the series expansion of $\tilde{h}$  in a neighborhood of 0, we have
$$
 z^p \tilde{h}^{(p)}(z)=\sum_{k=p}^{\infty}k\ldots(k-p+1)a_kz^k
$$
and therefore
\begin{equation}\label{ed1}
\sum_{p=0}^{m-1} \alpha_p (z+z_1)^{m-1-p} (z+\beta)^p\sum_{k=p}^{\infty}%\underbrace{
k\ldots(k-p+1)
%}_{p\text{ terms}}
a_kz^k=H(z+z_1).
\end{equation}
%Let $n\geq 2m-1$. 
We compute the coefficient of $z^n$ in the left-hand side of (\ref{ed1}) which equals
$$
%{1\over n!}{\partial^n\over\partial z^n}
\frac{1}{n!}\left(\sum_{p=0}^{m-1} \alpha_p (z+z_1)^{m-1-p} (z+\beta)^p\sum_{k=p}^{\infty}a_k \frac{k!}{(k-p)!}z^k\right)^{(n)}(0).
$$
By Leibniz rule we get
$$
\frac{1}{n!}\sum_{p=0}^{m-1}\alpha_{p}
%(-1)^{m-p+1}{(2m-p-2)!\over (m-p-1)!p!}\times\qquad\qquad\qquad\qquad\qquad\qquad\qquad\qquad\qquad\qquad\qquad\qquad\qquad
%$$
%$$
%\quad\qquad\qquad\times
\sum_{k_1+k_2+k_3=n}{1\over k_1!k_2!k_3!}{(m-1-p)!\over(m-1-p-k_1)!} z_1^{m-1-p-k_1} {p!\over (p-k_2)!}\beta^{p-k_2} a_{k_3}{ (k_3!)^2\over (k_3-p)!},
$$
where we adopt the convention that ${1/q!}=0$ if $q$ is a negative integer.
Taking into account the definition of $\alpha_{p}$, this simplifies to
$$
\frac{1}{n!}\sum_{p=0}^{m-1}(-1)^p(2m-p-2)!\left({\beta\over z_1}\right)^p\sum_{k_1+k_2+k_3=n}{k_3!z_1^{-k_1}\beta^{-k_2}\over k_1!k_2!(m-1-p-k_1)!(p-k_2)!(k_3-p)!}a_{k_3}.
$$
Let us write $k_3=n-\ell$ where $\ell$ is a new index. Then $k_1+k_2=\ell$ in the previous sum. Since $k_1\leq m-p-1$ and $k_2\leq p$, we deduce that $\ell$ takes its values in the set $\{0,\ldots, m-1\}$.
We thus obtain
$$
\frac{1}{n!}\sum_{\ell=0}^{m-1}(n-\ell)!a_{n-\ell}\sum_{p=0}^{m-1}{(-1)^p(2m-p-2)!\over (n-\ell-p)!}\left({\beta\over z_1}\right)^p\sum_{k_1+k_2=\ell}
{z_1^{-k_1}\beta^{-k_2}\over k_1!k_2!(m-1-p-k_1)!(p-k_2)!},
$$
or writing $k_1=k-p$ and $k_2=p+\ell-k$ in terms of a new index $k$, 
$$
\frac{1}{n!}\sum_{\ell=0}^{m-1}a_{n-\ell}{(n-\ell)!\over\beta^\ell}\sum_{p=0}^{m-1}(-1)^p{(2m-p-2)!\over (n-\ell-p)!}\sum_{k=0}^{m-1}{({\beta/z_1})^k\over (k-p)!(\ell+p-k)!(m-1-k)!(k-\ell)!}.
$$
Interchanging the two last sums, we obtain
$$
\sum_{\ell=0}^{m-1}a_{n-\ell}\beta_{\ell}^{n}
%{(n-\ell)!\over\beta^\ell}\sum_{k=0}^{m-1}
%{({\beta/z_1})^k\over (m-1-k)!(k-\ell)!}\sum_{p=0}^{m-1} {(-1)^p (2m-p-2)!\over (k-p)!(n-\ell-p)!(\ell+p-k)!}
,\qquad n\geq 0,%2m-1,
$$
where $a_{n}=0$, $n<0$, and we have set, for $\ell\in\{0,\ldots,m-1\}$,% and $n\geq 2m-1$, 
\begin{equation}\label{beta}
\beta_\ell^n:=\frac{(n-\ell)!}{n!\beta^\ell}\sum_{k=0}^{m-1}{({\beta/ z_1})^k\over (m-1-k)!(k-\ell)!}\sum_{p=0}^{m-1} {(-1)^p (2m-p-2)!\over (k-p)!(n-\ell-p)!(\ell+p-k)!}.
\end{equation}
%then
%$$
%\sum_{\ell=0}^{m-1}a_{n-\ell}\beta_\ell^n=0.
%$$
If $\ell=0$, the second sum in (\ref{beta}) equals
$$
\sum_{p=0}^{m-1} \frac{(-1)^p(2m-p-2)!}{(k-p)!(n-p)!(p-k)!}=
(-1)^{k}\frac{(2m-k-2)!}{(n-k)!},%\neq 0,
$$
and
$$\beta_{0}^{n}=\sum_{k=0}^{m-1}\frac{(-\beta/z_{1})^{k}(2m-k-2)!}
{(m-1-k)!k!(n-k)!}\neq 0,$$
since $\beta/z_{1}<0$. The nonvanishing of the coefficients $\beta_{0}^{n}$ implies that each coefficient $a_{n}$, $n>0$, is uniquely determined whence the first assertion of the lemma.

%We have thus shown (\ref{fact2}).
Next, in order to prove the vanishing of all coefficients $a_{n}$, $n\geq 2m-1$, when the right-hand side is a polynomial of degree at most $2m-2$, it is sufficient to show that
\begin{equation}\label{fact1}
\beta_{N+1}^n=\cdots=\beta_{m-1}^n=0,\quad N\in\{0\ldots,m-2\}, \quad n=2m-1+N.
\end{equation}
%and
%\begin{equation}\label{fact2}
%\beta_{0}^{n}\neq 0,\quad n\geq 2m-1.
%\end{equation}
%We first prove (\ref{fact1}).
Let $N$ be fixed in $\{0\ldots,m-2\}$ and $\ell\in\{N+1,\ldots,m-1\}$. From our convention on the factorials of negative integers, the second sum in (\ref{beta}) actually simplifies to
\begin{align*}
%\sum_{p=0}^{m-1}{(-1)^p (2m-p-2)!\over (k-p)!(n-\ell-p)!(\ell+p-k)!}=
\sum_{p=k-\ell}^k{(-1)^p (2m-p-2)!\over (k-p)!(n-\ell-p)!(\ell+p-k)!}
& =(-1)^{k-\ell}\sum_{p=0}^\ell{(-1)^p (2m-k+\ell-p-2)!\over (\ell-p)!p!(n-k-p)!}
\\
& =(-1)^{k-\ell}\sum_{p=0}^\ell{(2m-k+\ell-p-2)!\over(2m-1+N-k-p)!}{(-1)^p\over p!(\ell-p)!}.
\end{align*}
The last expression equals the derivative of order $\ell-(N+1)$ of the polynomial
$$
(-1)^{N+1-\ell}\sum_{p=0}^\ell{x^{2m-k+\ell-p-2}\over p!(\ell-p)!}
$$
taken at $x=-1$. Thanks to the binomial identity, this polynomial rewrites as
$$
(-1)^{N+1-\ell}x^{2m-k-2}(1+x)^\ell/\ell!.
$$
We then get
$$
\sum_{p=0}^{m-1}{(-1)^p (2m-p-2)!\over (k-p)!(n-\ell-p)!(\ell+p-k)!}=0
$$
which shows (\ref{fact1}).
%$$
%\beta_{N+1}^n=\cdots=\beta_{m-1}^n=0.
%$$
\end{proof}
\subsection{Case of an even positive integer coefficient $\alpha$}
Here we assume $\alpha=2(m+1)$, $m\in\N$. First, let us check that the global relation (\ref{rel-glob}) does not allow the reconstruction of $u_{n}$ in this case. Indeed, integrating on $\T$ instead of $\CC_{a}$, the part involving $u_{n}$ in (\ref{rel-glob}) now equals
\begin{equation}\label{rel-glob-pos-pair}
\int_{\T}\frac{(z-(k-a))^{m}(z+1/(k+a))^{m}f(z)}{z^{m+1}}dz
%\\=
%\int_{\T}\frac{(z-(k-a))^{m}(z+1/(k+a))^{m}(y+ik)u_{t}(z+a)}
%{z^{m+1}}dz,
\end{equation}
where $f$ is defined by (\ref{def-f}) and $k\in\C\setminus(\DD_{a}\cup\DD_{-a})$. With the definitions of $g_{1}, \mu, \varphi$ and $\A$ as in the previous case, we let, for $\mu\in\A$,
\begin{equation}\label{def-phi-pos}
\Phi_{1}(\mu):=\int_{\T}\frac{(z-\mu)^{m}(1-\varphi(\mu)z)^{m}g_{1}(z)}{z^{m+1}}dz
=\frac{2i\pi}{m!}((z-\mu)^{m}(1-\varphi(\mu)z)^{m}g_{1}(z))^{(m)}(0),%\qquad\mu\in\A,
\end{equation}
%It is clear that $\Phi$ is analytic in $\A$ and extends analytically to $\D$ except at $\mu=-1/2a$ where it has a pole of order at most $m$. 
so that, as before, (\ref{rel-glob-pos-pair}) can be rewritten as
%\begin{equation}\label{rel-glob4-pos}
$\Phi_{1}(\mu)-\bar{\Phi}_{1}(\varphi(\mu))$. Since this expression only involves the derivatives of $g_{1}$ at 0 of order up to $m$, it does not contain enough information to reconstruct $g_{1}(z)$ and a fortiori the function $f(z)$.
%$=\Psi(\mu),\qquad\mu\in\A,
%\end{equation}
%where
%$$\Psi(\mu):=\frac{1}{\mu}\int_{\T}
%\frac{(z-\mu)^{m}(1-\varphi(\mu)z)^{m}(\mu y-i(a\mu+1))u_{t}(z+a)}{z^{m+1}}dz.$$
\begin{proof}[Proof of Theorem \ref{corresp} when $\alpha=2(m+1)$, $m\in\N$]
Expressing the closed differential form (\ref{defW2}) in terms of the directional derivatives, thanks to (\ref{polar}), and integrating along $\CC_{a}$, we now obtain
$$\int_{\CC_{a}}
\frac{((y+ik)xu_{t}(z)-x^{2}u_{n}(z)-(2m+1)(z\bar z+iyk)u(z))x^{2m}}
{(k-z)^{m+1}(k+\bar{z})^{m+1}}ds=0,$$
for $k\in\C\setminus(\DD_{a}\cup\DD_{-a})$.
Since we assume that $u$ and $u_{t}$ are known on the boundary $\CC_{a}$ of the domain, our problem is to recover $u_{n}$ from the knowledge, for $k\in\C\setminus(\DD_{a}\cup\DD_{-a})$, of the integral
$$\int_{\CC_{a}}\frac{x^{2m+2}u_{n}(z)}{(k-z)^{m+1}(k+\bar{z})^{m+1}}ds,$$
which is precisely what was done in the previous case (except that the definition (\ref{def-f}) of the function $f$ there has to be replaced with $f(z)=(x+a)^{2m+2}u_{n}(z+a)$).
\end{proof}

\obeylines
\texttt{

\medskip

\medskip
Slah Chaabi, Stephane Rigat, Franck Wielonsky
slah.chaabi, stephane.rigat, franck.wielonsky@univ-amu.fr
Laboratoire I2M - UMR CNRS 7373
Universit\'e Aix-Marseille
CMI 39 Rue Joliot Curie
F-13453 Marseille Cedex 20, FRANCE
}

\end{document}